\documentclass[a4paper,oneside]{amsart}
\usepackage[utf8]{inputenc}
\usepackage[a4paper,margin=3cm]{geometry} 
\usepackage{amsmath}
\usepackage{amsthm}
\usepackage{amscd}
\usepackage{amssymb}
\usepackage{MnSymbol}
\usepackage{latexsym}
\usepackage{eucal}
\usepackage{dsfont}
\usepackage{mathtools}
\usepackage{enumitem}

\usepackage{verbatim}

\usepackage[colorlinks,pdftex]{hyperref}
\hypersetup{
linkcolor=black,
citecolor=black,
pdftitle={},
pdfauthor={Franziska K\"uhn},
pdfkeywords={},
}

\makeatletter
\renewcommand\section{\@startsection{section}{1}{0mm}{-1.5\baselineskip}{\baselineskip}{\normalsize\bfseries\sffamily}}
\renewcommand\subsection{\@startsection{subsection}{1}{0mm}{-\baselineskip}{\baselineskip}{\normalsize\bfseries\sffamily}}
\makeatother

\makeatletter
\def\@fnsymbol#1{\ensuremath{\ifcase#1\or *\or **\or \dagger\or \ddagger\or
   \mathsection\or \mathparagraph\or \|\or \dagger\dagger
   \or \ddagger\ddagger \else\@ctrerr\fi}}

\newlength{\preskip}
\newlength{\postskip}
\makeatother

\newtheoremstyle{theorem}{\preskip}{\postskip}{\itshape}{}{\bfseries}{}
{.5em}{\textbf{\thmname{#1}\thmnumber{ #2} (\thmnote{ #3})}}
\newtheoremstyle{definition}{\preskip}{\postskip}{\normalfont}{0pt}{\bfseries}{}{.5em}{}
\newtheoremstyle{remark}{\preskip}{\postskip}{\normalfont}{0pt}{\bfseries}{}{.5em}{}

\swapnumbers
\theoremstyle{theorem} \newtheorem{thm}{Theorem}[section]
\theoremstyle{theorem} \newtheorem{lem}[thm]{Lemma}
\theoremstyle{theorem} \newtheorem{prop}[thm]{Proposition}
\theoremstyle{theorem} \newtheorem{kor}[thm]{Corollary}
\theoremstyle{definition} \newtheorem{defn}[thm]{Definition}
\theoremstyle{remark} \newtheorem{bem_thm}[thm]{Remark}
\theoremstyle{remark} 
\theoremstyle{definition} 
\theoremstyle{definition} 
\theoremstyle{remark} 
\theoremstyle{remark} 
\theoremstyle{definition}  
\theoremstyle{definition}  
\theoremstyle{definition}

\DeclareMathOperator \id {id}

\DeclareMathOperator \spt {supp}

\DeclareMathOperator \tr {tr}

\newcommand{\I}{\mathds{1}}

\newcommand\fa{\qquad \text{for all \ }}

\newcommand{\cadlag}{c\`adl\`ag }

\newcommand\mc[1] {\mathcal{#1}}
\newcommand\mbb[1] {\mathds{#1}}

\newcommand{\eps}{\varepsilon}
\newcommand{\ac}{\text{ac}}
\newcommand{\sem}{\text{sem}}

\begin{document}

\title[Solutions to HJB equations associated with sublinear L\'evy(-type) processes]{Viscosity solutions to Hamilton--Jacobi--Bellman equations associated with sublinear L\'evy(-type) processes}
\author[F.~K\"{u}hn]{Franziska K\"{u}hn} 
\address[F.~K\"{u}hn]{Institut de Math\'ematiques de Toulouse, Universit\'e Paul Sabatier III Toulouse, 118 Route de Narbonne, 31062 Toulouse, France}
\email{franziska.kuhn@math.univ-toulouse.fr}
\begin{abstract}
	Using probabilistic methods we study the existence of viscosity solutions to  non-linear integro-differential equations \begin{align*} 
	\partial_t u(t,x) - \sup_{\alpha \in I} &\bigg( b_{\alpha}(x) \cdot \nabla_x u(t,x) + \frac{1}{2} \tr\left(Q_{\alpha}(x) \cdot \nabla^2_x u(t,x)\right) \\
	& +\int_{y \neq 0} \big(u(t,x+y)-u(t,x)-\nabla_x u(t,x) \cdot h(y) \big) \, \nu_{\alpha}(x,dy) \bigg) = 0
	\end{align*}
	with initial condition $u(0,x)= \varphi(x)$; here $(b_{\alpha}(x),Q_{\alpha}(x),\nu_{\alpha}(x,dy))$, $\alpha \in I$, $x \in \mathbb{R}^d$, is a family of L\'evy triplets and $h$ is some truncation function. The solutions, which we construct, are of the form $u(t,x) = T_t \varphi(x)$ for a sublinear Markov semigroup $(T_t)_{t \geq 0}$ with representation 
	\begin{equation*}
	T_t \varphi(x) =  \mathcal{E}^x \varphi(X_t):= \sup_{\mathbb{P} \in \mathfrak{P}_x} \int_{\Omega} \varphi(X_t) \, d\mathbb{P}
	\end{equation*}
	where $(X_t)_{t \geq 0}$ is a stochastic process and $\mathfrak{P}_x$, $x \in \mathbb{R}^d$, are families of probability measures. The key idea is to exploit the connection between sublinear Markov semigroups and the associated Kolmogorov backward equation. In particular, we obtain new existence and uniqueness results for viscosity solutions to Kolmogorov backward equations associated with L\'evy(-type) processes for sublinear expectations and Feller processes on classical probability spaces.
\end{abstract}
\maketitle

\section{Introduction}

Markov processes and the semigroups generated by them play an important role in the study of evolution equations. For a Markov process $(X_t)_{t \geq 0}$ with semigroup $T_t \varphi(x) := \mbb{E}^x \varphi(X_t)$ it is well-known that -- under suitable assumptions -- the mapping $u(t,x) := T_t \varphi(x)$ is a solution to the Kolmogorov backward equation \begin{equation}
	\frac{\partial}{\partial t} u(t,x) - A_x u(t,x) = 0 \qquad u(0,x) = \varphi(x) \label{intro-eq1}
\end{equation}
where $A=A_x$ is the infinitesimal generator of the semigroup $(T_t)_{t \geq 0}$. For instance if $\varphi \in C_c^{\infty}(\mbb{R}^d)$ is a compactly supported smooth function and $(X_t)_{t \geq 0}$ is a L\'evy process \cite{sato} or a ``good'' Feller process \cite{ltp,jac2} with infinitesimal generator $A$, then $u(t,x) = \mbb{E}^x \varphi(X_t)$ solves the evolution equation \eqref{intro-eq1}, and $A|_{C_c^{\infty}(\mbb{R}^d)}$ is a pseudo-differential operator with a representation of the form \begin{equation*}
	Af(x) = b(x) \cdot \nabla f(x) + \frac{1}{2} \tr\left(Q(x) \cdot \nabla^2 f(x)\right) + \int_{y \neq 0} \big( f(x+y)-f(x)-\nabla f(x) \cdot h(y) \big) \, \nu(x,dy)
\end{equation*}
where $h$ is some truncation function and $(b(x),Q(x),\nu(x,dy))$ is a L\'evy triplet for each fixed $x \in \mbb{R}^d$, see Section~\ref{pre} for details.  \par
In this paper, we are interested in evolution equations for \emph{sub}linear Markov semigroups, that is, semigroups $(T_t)_{t \geq 0}$ of sublinear operators $T_t$, see Definition~\ref{main-0} for the precise definition. They appear naturally in the study of Markov processes on sublinear expectation spaces; quite often $(T_t)_{t \geq 0}$ has a representation of the form \begin{equation*}
	T_t \varphi(x) = \mc{E}^x \varphi(X_t) := \sup_{\mbb{P} \in \mathfrak{P}_x} \int_{\Omega} \varphi (X_t) \, d\mbb{P}
\end{equation*}
for families of probability measures $\mathfrak{P}_x$, $x \in \mbb{R}^d$. As in the case of classical Markov semigroups, it is possible to associate an evolution equation \eqref{intro-eq1} with a sublinear Markov semigroup $(T_t)_{t \geq 0}$, cf.\ Hollender \cite{julian}. In this paper we exploit the connection between sublinear Markov semigroups and the associated evolution equation to establish new existence and uniqueness results for solutions to non-local non-linear Hamilton--Jacobi--Bellman (HJB) equations \begin{align} \label{intro-eq3} \begin{aligned}
\partial_t u(t,x) -  \sup_{\alpha \in I} &\bigg( b_{\alpha}(x) \cdot \nabla_x u(t,x) + \frac{1}{2} \tr\left(Q_{\alpha}(x) \cdot \nabla^2_x u(t,x)\right) \\
& +\int_{y \neq 0} \big(u(t,x+y)-u(t,x)-\nabla_x u(t,x) \cdot h(y) \big) \, \nu_{\alpha}(x,dy) \bigg) = 0\end{aligned}
\end{align}
using probabilistic methods; here $I$ is an index set and $(b_{\alpha}(x),Q_{\alpha}(x),\nu_{\alpha}(x,dy))$ is a L\'evy triplet for each $x \in \mbb{R}^d$ and $\alpha \in I$. In particular, we will see that HJB equations  \eqref{intro-eq3} appear in probability theory as Kolmogorov backward equations of Markov processes for sublinear expectations. The stochastic processes, which we study in this paper, can be interpreted as generalizations of classical L\'evy-type processes under uncertainty in their semimartingale characteristics. \par
Non-linear integro-partial differential equations \eqref{intro-eq3} have attracted a lot of attention in the last years. In particular, there has been a substantial progess in extending the highly developped viscosity solution theory for second order \emph{local} equations to second order \emph{non-local} equations. Since there is a large amount of literature, let us just mention a few fundamental results. Important contributions to comparison principles were obtained by Alvarez \& Tourin \cite{alvarez}, Barles \& Imbert \cite{barles07} and Jakobsen \& Karlsen \cite{jak04}. These works use mostly analytical approaches, and they differ in the exact form of the admissible equations, the behaviour of the solutions at infinity as well as in the assumptions on the continuity of the coefficients $(b_{\alpha}(x),Q_{\alpha}(x),\nu_{\alpha}(x,dy))$ and on the singularities of the measures $\nu_{\alpha}(x,dy)$. Recently, Hollender \cite{julian} succeeded in relaxing the assumptions from \cite{jak04} and obtained a quite general comparison principle. Barles, Chasseigne \& Imbert \cite{barles10} study the regularity of viscosity solutions with respect to the space variable $x$, and \cite{jak05} investigates the continuous dependence on the coefficients $(b_{\alpha}(x),Q_{\alpha}(x),\nu_{\alpha}(x,dy))$. There are close connections between HJB equations and several areas of probability theory, e.\,g.\ backwards stochastic differential equations \cite{barles97} and stochastic control with jumps, cf.\ \cite{biswas,pham} and the references therein.  \par
The approach which we use in this paper goes back to Shige Peng; he was one of the first researchers to investigate Kolmogorov equations associated with sublinear Markov processes. In his pioneering work \cite{peng} Peng constructed the so-called G-Brownian motion: a continuous stochastic process with independent and stationary increments whose associated evolution equation is the G-heat equation\begin{equation*}
	\partial_t u(t,x)- \sup_{\alpha \in I} \left( \frac{1}{2} \tr(Q_{\alpha} \cdot \nabla^2_x u(t,x)) \right)=0 
\end{equation*}
for a given index set $I$ and a family $Q_{\alpha}$, $\alpha \in I$, of positive semi-definite symmetric matrices. His approach was generalized in Hu \& Peng \cite{hu} who constructed a class of stochastic processes with independent and stationary increments such that the associated Kolmogorov backward equation is given by the non-local non-linear equation \begin{align} \label{intro-eq7} \begin{aligned}
	\partial_t u(t,x) - \sup_{\alpha \in I} \bigg( b_{\alpha} \cdot &\nabla_x u(t,x) + \frac{1}{2} \tr(Q_{\alpha} \cdot \nabla_x^2 u(t,x)) \\
	&\quad+ \int_{y \neq 0} \big( u(t,x+y)-u(t,x)-\nabla_x u(t,x) \cdot h(y) \big) \, \nu_{\alpha}(dy) \bigg)=0 \end{aligned}
\end{align}
for a family $(b_{\alpha},Q_{\alpha},\nu_{\alpha}(dy))$, $\alpha \in I$, of L\'evy triplets. In their construction it is a-priori unclear how one should define the associated semigroup $T_t \varphi(x)$ for functions $\varphi$ which are not continuous. Neufeld \& Nutz \cite{nutz} studied sublinear Markov semigroups of the form \begin{equation*}
	T_t \varphi(x) := \mc{E}\varphi(x+X_t) := \sup_{\mbb{P}} \int_{\Omega} \varphi(x+X_t) \, d\mbb{P}
\end{equation*}
where the supremum is taken over probability measures $\mbb{P}$ such that $(X_t)_{t \geq 0}$ is a $\mbb{P}$-semimartingale with differential characteristics taking value in a given family $(b_{\alpha},Q_{\alpha},\nu_{\alpha})$, $\alpha \in I$, of L\'evy triplets; they showed that $u(t,x) := T_t \varphi(x)$ is a viscosity solution to \eqref{intro-eq7} under the assumption that $\varphi$ is Lipschitz continuous and bounded, and \begin{equation*}
	\sup_{\alpha \in I} \left( |b_{\alpha}| + |Q_{\alpha}| + \int_{y \neq 0} \min\{|y|,|y|^2\} \, \nu_{\alpha}(dy) \right)<\infty. 
\end{equation*}
Using a version of Kolmogorov's extension theorem for non-linear expectations, Denk et al.\ \cite{denk} recently established under the weaker condition 
\begin{equation*}
	\sup_{\alpha \in I} \left( |b_{\alpha}| + |Q_{\alpha}| + \int_{y \neq 0} \min\{1,|y|^2\} \, \nu_{\alpha}(dy) \right)<\infty
\end{equation*}
the existence of a viscosity solution to \eqref{intro-eq7} for initial conditions $u(0,x) = \varphi(x)$ which are bounded and uniformly continuous; we will recover this statement as a particular case of our main result, cf.\ Corollary~\ref{main-7}. \par
This paper builds on the PhD thesis \cite{julian} by Hollender who obtained many new insights on Markov processes for sublinear expectations. In particular, he generalized the approach by Nutz \& Neufeld \cite{nutz} to HJB equations \eqref{intro-eq3} with state-space dependent coefficients $(b_{\alpha}(x),Q_{\alpha}(x), \nu_{\alpha}(x,dy))$, and succeeded in constructing a class of sublinear Markov processes with associated evolution equation \eqref{intro-eq3}. Compared to \cite{julian}, our results require weaker regularity and integrability assumptions, see the discussion in Section~\ref{main} for details. The key tool to relax the assumptions from \cite{julian} is a maximal inequality which allows us to estimate expressions of the form \begin{equation*}
	\sup_{\mbb{P} \in \mathfrak{P}_x} \mbb{P} \left( \sup_{s \leq t} |X_s-x| > r\right)
\end{equation*}
for a family of probability measures $\mathfrak{P}_x$, cf.\ Section~\ref{max}. \par \medskip

This article is organized as follows. After introducing basic notation and definitions in Section~\ref{pre}, we recall some results on sublinear Markov semigroups at the beginning of Section~\ref{main}. In Section~\ref{main} we also state and discuss our main result, Theorem~\ref{main-3}. Several applications of Theorem~\ref{main-3} will be presented in Section~\ref{ex}; we will study Kolmogorov backward equations associated with L\'evy and L\'evy-type processes for sublinear expectations and evolution equations associated with classical Feller processes. In Section~\ref{max} we will establish the maximal inequality which is a crucial tool for the proofs which will be presented in Section~\ref{p}.

\section{Preliminaries} \label{pre} 

We consider the Euclidean space $\mbb{R}^d$ endowed with the Borel $\sigma$-algebra $\mc{B}(\mbb{R}^d)$ and write $B(x,r)$ for the open ball centered at $x \in \mbb{R}^d$ with radius $r>0$. If $A \in \mbb{R}^{d \times d}$ is a matrix, then $A^T$ is the transpose of $A$ and $\tr(A)$ is the trace of $A$. The gradient and Hessian of a function $f: \mbb{R}^d \to \mbb{R}$ are denoted by $\nabla f$ and $\nabla^2 f$, respectively, and on $C_b^2(\mbb{R}^d)$, the space of functions with bounded derivatives up to order $2$, we define a norm by  \begin{equation*}
	\|f\|_{(2)} := \|f\|_{\infty} + \|\nabla f\|_{\infty} + \|\nabla^2 f\|_{\infty}. 
\end{equation*}
The space of bounded Borel measurable functions $f: \mbb{R}^d \to \mbb{R}$ is denoted by $\mc{B}_b(\mbb{R}^d)$. A function $f: [0,\infty) \to \mbb{R}^d$ is in the Skorohod space $D([0,\infty), \mbb{R}^d)$ if $f$ is c\`adl\`ag, i.\,e.\ $f$ is right-continuous and has finite left-hand limits in $\mbb{R}^d$.  We will use the shorthand \begin{equation*}
	D_x := D_x([0,\infty), \mbb{R}^d) := \{f \in D([0,\infty), \mbb{R}^d); f(0)=x\}.
\end{equation*}
For a probaility measure $\mbb{P}$ we denote by $\mbb{E}:=\mbb{E}_{\mbb{P}} := \int \, d\mbb{P}$ the expectation with respect to $\mbb{P}$. We use $\lambda$ to denote the Lebesgue measure. Throughout this paper, we denote by $h$ a truncation function, i.\,e.\ a bounded mapping $h:\mbb{R}^d \to \mbb{R}^d$ with bounded support such that $h(x)=x$ in a neighborhood of $0$. \par
Let $(\Omega,\mc{A},\mbb{P})$ be a probability space. If $B_t: \Omega \to \mbb{R}^d$ is a predictable \cadlag bounded variation process, $C_t: \Omega \to \mbb{R}^{d \times d}$ a continuous bounded variation process and $F$ a predictable random measure on $[0,\infty) \times \mbb{R}^d$, then the triplet $(B,C,F)$ is called the \emph{(predictable) semimartingale characteristics} (with respect to a truncation function $h$) of a semimartingale $(X_t)_{t \geq 0}$ if the process \begin{align*}
	f(X_t)-f(X_0)-& \sum_{j=1}^d \int_0^t \partial_{x_j} f(X_{s-}) \, dB_s^j - \sum_{i,j=1}^d \int_0^t \partial_{x_i} \partial_{x_j} f(X_{s-}) \, dC_s^{i,j} \\
	& - \int_0^t\!\!\int_{y \neq 0} \left( f(X_{s-}+y)-f(X_{s-})- \nabla f(X_{s-}) \cdot h(y) \right) \, F(dy,ds)
\end{align*}
is a local martingale for any $f \in C_b^2(\mbb{R}^d)$. If the triplet $(B,C,F)$ is absolutely continuous with respect to Lebesgue measure, in the sense that \begin{equation*}
	dB_t = b_t \, dt \quad dC_t = Q_t \, dt \quad F(dy,dt) = \nu_t(dy) \, dt
\end{equation*}
for predictable processes $(b_t)_{t \geq 0}$, $(Q_t)_{t \geq 0}$ and a predictable kernel $\nu$, then we call $(b_t,Q_t,\nu_t)$ the \emph{differential characteristics} of $(X_t)_{t \geq 0}$; we write $(b_t^{\mbb{P}},Q_t^{\mbb{P}},\nu_t^{\mbb{P}})$ if we need to emphasize the underlying probability measure $\mbb{P}$. We denote by $\mathfrak{P}_{\sem}^{\ac}(\Omega)$ the family of probability measures $\mbb{P}$ on $\Omega$ such that $(X_t)_{t \geq 0}$ is a semimartingale (with respect to $\mbb{P}$) which has absolutely continuous semimartingale characteristics with respect to Lebesgue measure. For a thorough discussion of semimartingales and their characteristics we refer to Jacod \cite{jacod}.   \par
A \emph{pseudo-differential operator with continuous negative definite symbol} is an operator $A: C_c^{\infty}(\mbb{R}^d) \to \mbb{R}$ of the form 
\begin{equation}
	A f(x)= -q(x,D) f(x) := - \int_{\mbb{R}^d} q(x,\xi) e^{ix \cdot \xi} \hat{f}(\xi) \, d\xi, \qquad x \in \mbb{R}^d, \, f \in C_c^{\infty}(\mbb{R}^d) \label{pseudo}
\end{equation}
where $\hat{f}(\xi) = (2\pi)^{-d} \int_{\mbb{R}^d} f(x) e^{-ix \cdot \xi} \, dx$ denotes the Fourier transform of $f$ and the \emph{symbol} $q(x,\xi)$ is a continuous negative definite function for each $x \in \mbb{R}^d$, i.\,e.\ 
\begin{equation}
	q(x,\xi) = -ib(x) \cdot \xi + \frac{1}{2} \xi \cdot Q(x) \xi + \int_{y \neq 0} \left( 1-e^{iy \cdot \xi} + i \xi \cdot h(y) \right) \, \nu(x,dy), \quad x,\xi \in \mbb{R}^d, \label{sym}
\end{equation}
here $h$ is a given truncation function and $(b(x),Q(x),\nu(x,dy))$ is for each $x \in \mbb{R}^d$ a \emph{L\'evy triplet} consisting of a vector $b(x) \in \mbb{R}^d$ (\emph{drift vector}), a symmetric positive semidefinite matrix $Q(x) \in \mbb{R}^{d \times d}$ (\emph{diffusion matrix}) and a \emph{L\'evy measure} $\nu(x,dy)$, i.\,e.\ measure on $(\mbb{R}^d \backslash \{0\}, \mc{B}(\mbb{R}^d \backslash \{0\}))$ satisfying $\int_{y \neq 0} \min\{1,|y|^2\} \, \nu(x,dy)<\infty$. We call $(b,Q,\nu)$ is the \emph{characteristics} of $q$. For a fixed truncation function $h$ the characteristics $(b,Q,\nu)$ is uniquely determined by $q$; note that only the drift $b(x)$ depends on the choice of $h$. Using properties of the Fourier transform it follows readily that 
\begin{align*}
	Af(x) &= b(x) \cdot \nabla f(x) + \frac{1}{2} \tr\left(Q(x) \cdot \nabla^2 f(x)\right) + \int_{y \neq 0} \left( f(x+y)-f(x)- \nabla f(x) \cdot h(y) \right) \, \nu(x,dy),
\end{align*}
and therefore $A$ extends naturally to $C_b^2(\mbb{R}^d)$. Pseudo-differential operators with negative definite symbol play an important role in the study of Feller processes, see e.\,g.\ the monograps \cite{ltp,jac2,matters} for details, and in the context of stochastic differential equations, see e.\,g.\ \cite{sde,kurtz}. If the characteristics $(b,Q,\nu)$ (and hence the symbol $q$) does not depend on $x$, then $q$ is the characteristic exponent of a \emph{L\'evy process}, i.\,e.\ a stochastic process with \cadlag sample paths and independent and stationary increments, and $(b,Q,\nu)$ is its L\'evy triplet with respect to the truncation function $h$, cf.\ Sato \cite{sato} and also Khoshnevisan \& Schilling \cite{barca}. \par
In this paper we study Hamilton--Jacobi--Bellman equations of the form \begin{equation}
	\frac{\partial}{\partial t} u(t,x)- \sup_{\alpha \in I} A^{\alpha}_x u(t,x) = 0 \label{hjb}
\end{equation}
where $A^{\alpha}$ is for each $\alpha \in I$ a pseudo-differential operator; we write $A^{\alpha}_x$ to emphasize that $A^{\alpha}$ acts with respect to the space variable $x$. In general, there do not exist classical solutions to \eqref{hjb}. We will work with the weaker notion of viscosity solutions which was originally introduced by Crandall \& Lions \cite{lions} and Evans \cite{evans}. The following definition is taken from \cite{julian}; for a discussion of equivalent definitions we refer the reader to \cite[Chapter 2]{julian} and \cite{barles07}.

\begin{defn} \label{pre-21}
	Let $A: \mc{D}(A) \to \mbb{R}$ be an operator with domain $\mc{D}(A)$ containing the space of smooth functions with bounded derivatives $C_b^{\infty}(\mbb{R}^d)$. An upper semicontinuous function $u: [0,\infty) \times \mbb{R}^d \to \mbb{R}$ is a \emph{viscosity subsolution} to the equation \begin{equation*}
	\partial_t u(t,x)- A_x u(t,x)=0 
	\end{equation*}
	if the inequality $\partial_t \varphi(t,x)-A_x \varphi(t,x) \leq 0$ holds for any function $\varphi \in C_b^{\infty}([0,\infty) \times \mbb{R}^d)$ such that $u-\varphi$ has a global maximum in $(t,x) \in (0,\infty) \times \mbb{R}^d$ with $u(t,x) = \varphi(t,x)$. A mapping $u$ is a \emph{viscosity supersolution} if $-u$ is a viscosity subsolution. If $u$ is both a viscosity sub- and supersolution, then $u$ is called \emph{viscosity solution}.
\end{defn}

\section{Main result} \label{main}

In order to state our main result, Theorem~\ref{main-3}, we first need to introduce Markov sublinear semigroups and their associated evolution equation.

\begin{defn}  \label{main-0} 
	Let $\mc{H}$ be a convex cone of functions $f: \mbb{R}^d \to \mbb{R}$ containing all constant functions. A family of sublinear operators $T_t: \mc{H} \to \mc{H}$, $t \geq 0$, is a \emph{sublinear Markov semigroup (on $\mc{H}$)} if it satisfies the following properties. \begin{enumerate}
		\item $(T_t)_{t \geq 0}$ has the semigroup property, i.\,e.\ $T_{t+s} = T_t T_s$ for all $s,t \geq 0$ and $T_0 = \id$,
		\item $T_t$ is monotone for each $t \geq 0$,  i.\,e.\ $f,g \in \mc{H}$, $f \leq g$ implies $T_tf\leq T_tg$,
		\item $T_t$ preserves constants for each $t \geq 0$, i.\,e.\ $T_t(c)=c$ for all $c \in \mbb{R}$.
	\end{enumerate}  
\end{defn}

Following \cite{julian} we associate the \emph{sublinear infinitesimal generator} $A: \mc{D}(A) \to \mc{H}$ with the sublinear Markov semigroup $(T_t)_{t \geq 0}$, \begin{equation*}
	Af(x) := \lim_{t \downarrow 0} \frac{T_t f(x)-f(x)}{t}, \qquad x \in \mbb{R}^d, f \in \mc{D}(A),
\end{equation*}
where \begin{equation*}
	\mc{D}(A) := \left\{f \in \mc{H}; \exists g \in \mc{H} \, \, \forall x \in \mbb{R}^d: \, \, g(x) = \lim_{t \downarrow 0} \frac{T_t f(x)-f(x)}{t} \right\}
\end{equation*}
is the domain of $A$. Hollender \cite[Proposition 4.10]{julian} established the following fundamental result which associates to the sublinear Markov semigroup $(T_t)_{t \geq 0}$ a evolution equation. It plays a key role in the proof of our main result.

\begin{thm} \label{pre-33}
	Let $(T_t)_{t \geq 0}$ be a sublinear Markov semigroup on a convex cone $\mc{H}$ of real-valued functions on $\mbb{R}^d$ containing all constants functions, and let $A: \mc{D}(A) \to \mc{H}$ be its generator. If $C_b^{\infty}(\mbb{R}^d) \subseteq \mc{D}(A)$ and if $f \in \mc{H}$ is such that $(t,x) \mapsto u(t,x) := T_t f(x)$ is continuous, then $u$ is a viscosity solution (in the sense of Definition~\ref{pre-21}) to  \begin{equation*}
		\partial_t u(t,x) - A_x u(t,x) = 0 
	\end{equation*}
	under $u(0,x)= f(x)$.
\end{thm}

Throughout the remaining part of this section, we use the canonical model, i.\,e.\ we denote by $(X_t)_{t \geq 0}$ the canonical process $X_t(\omega) := \omega(t)$, $\omega \in \Omega,$ on the Skorohod space $\Omega := D([0,\infty), \mbb{R}^d)$. Moreover, we write $\mathfrak{P}_{\sem}^{\ac}(D_x)$ for the family of probability measures $\mbb{P}$ on $D_x$ such that the canonical process $(X_t)_{t \geq 0}$ is a semimartingale (with respect to $\mbb{P}$) which has absolutely continuous semimartingale characteristics with respect to Lebesgue measure, cf.\ Section~\ref{pre}. In our main result, Theorem~\ref{main-3}, we study family of sublinear operators $(T_t)_{t \geq 0}$ of the form \begin{equation*}
	T_t f(x) = \sup_{\mbb{P} \in \mathfrak{P}_x} \mbb{E}_{\mbb{P}} f(X_t) 
\end{equation*}
where $\mathfrak{P}_x$ is for each $x \in \mbb{R}^d$ a set of probability measures on the Skorohod space satisfying
\begin{equation}
	\mathfrak{P}_x \subseteq \left\{ \mbb{P} \in \mathfrak{P}_{\sem}^{\ac}(D_x); (b_s^{\mbb{P}},Q_s^{\mbb{P}},\nu_s^{\mbb{P}})(\omega) \in \bigcup_{\alpha \in I} \{(b_{\alpha},Q_{\alpha},\nu_{\alpha})(X_s(\omega))\} \, \, \lambda(ds)\times \mbb{P}\text{-a.s.}\right\} \label{main-eq7}
\end{equation}
for a family of L\'evy triplets $(b_{\alpha}(x),Q_{\alpha}(x),\nu_{\alpha}(x,dy))$, $\alpha \in I$, $x \in \mbb{R}^d$, which is uniformly bounded on compact sets, i.\,e.\ 
\begin{equation}
	\forall r>0: \quad M_r:=\sup_{\alpha \in I} \sup_{|x| \leq r} \left( |b_{\alpha}(x)|+ |Q_{\alpha}(x)| + \int_{y \neq 0} \min\{|y|^2, 1\} \, \nu_{\alpha}(x,dy) \right) < \infty, \label{main-eq4}
\end{equation}
Following \cite{julian} we call the family of L\'evy triplets $(b_{\alpha}(x),Q_{\alpha}(x),\nu_{\alpha}(x,dy))$, $\alpha \in I$, $x \in \mbb{R}^d$, the \emph{uncertainty coefficients} and the families of probability measures $\mathfrak{P}_x$ \emph{uncertainty subsets}. In the sequel we will impose the following conditions.
\begin{enumerate}[label*=\upshape (C\arabic*),ref=\upshape C\arabic*] 
	\item\label{C4} $T_t f(x) := \mc{E}^x f(X_t) := \sup_{\mbb{P} \in \mathfrak{P}_x} \mbb{E}_{\mbb{P}} f(X_t)$ defines a sublinear Markov semigroup on a convex cone $\mc{H}$ of real-valued bounded functions containing all constant functions,
	\item\label{C5} For each $\alpha \in I$ there exists a measure $\mbb{P}^{\alpha} \in \mathfrak{P}_x$ with differential characteristics \begin{equation*}
		(b_s^{\mbb{P}_{\alpha}},Q_s^{\mbb{P}_{\alpha}},\nu_s^{\mbb{P}_{\alpha}}) := (b_{\alpha}(X_{s-}),Q_{\alpha}(X_{s-}),\nu_{\alpha}(X_{s-},\cdot)) \quad \lambda(ds)\times\mbb{P}\text{-a.s.} 
	\end{equation*}
	\item\label{C2} $K \ni x \mapsto b_{\alpha}(x)$, $\alpha \in I$, and $K \ni x \mapsto Q_{\alpha}(x)$, $\alpha \in I$, are uniformly equi-continuous,
	\item\label{C3} $K \ni x \mapsto \int_{y \neq 0} g(y) \, \nu_{\alpha}(x,dy)$, $\alpha \in I$, is uniformly equi-continuous for any $g \in C_b^1(\mbb{R}^d)$ which satisfies $|g(y)| \leq  \min\{1,|y|^2\}$, $y \in \mbb{R}^d$. 
\end{enumerate}
Let us mention that \eqref{C4} is automatically satisfied if $\mathfrak{P}_x$ equals the right-hand side of \eqref{main-eq7} and $(b_{\alpha}(x),Q_{\alpha}(x),\nu_{\alpha}(x))$ satisfies a certain measurability condition, cf.\ \cite[Remark 4.33]{julian}. The following statement is our main result. 

\begin{thm} \label{main-3}
	Let $h$ be a truncation function, and let $(b_{\alpha}(x),Q_{\alpha}(x),\nu_{\alpha}(x,\cdot))$, $\alpha \in I$, $x \in \mbb{R}^d$, be a family of L\'evy triplets which is uniformly bounded on compact sets (in the sense of \eqref{main-eq4}) and which satisfies at least one of the following two conditions. 
	\begin{enumerate}[label*=\upshape (A\arabic*),ref=\upshape A\arabic*] 
		\item\label{A1} Conditions \eqref{C2},\eqref{C3} hold for any compact set $K \subseteq \mbb{R}^d$, and the family $q_{\alpha}(x,\cdot): \mbb{R}^d \to \mbb{C}$, $\alpha \in I$, $x \in \mbb{R}^d$, of continuous negative definite functions associated with $(b_{\alpha}(x),Q_{\alpha}(x),\nu_{\alpha}(x,\cdot))$ via \eqref{sym} satisfies the uniform continuity condition \begin{equation*} 
			\lim_{r \to \infty} \sup_{|z-x| \leq r} \sup_{|\xi| \leq r^{-1}} \sup_{\alpha \in I} |q_{\alpha}(z,\xi)| =0 \fa x \in \mbb{R}^d.
		\end{equation*} 
		\item\label{A2} Conditions \eqref{C2},\eqref{C3} hold for $K=\mbb{R}^d$ and the family of L\'evy triplets is uniformly bounded: \begin{equation*}
			M := \sup_{\alpha \in I} \sup_{x \in \mbb{R}^d} \left( |b_{\alpha}(x)|+|Q_{\alpha}(x)| + \int_{y \neq 0} \min\{1,|y|^2\} \, \nu_{\alpha}(x,dy)\right)<\infty.
		\end{equation*}
	\end{enumerate} 
	Let $T_t f(x) := \mc{E}^x f(X_t) := \sup_{\mbb{P} \in \mathfrak{P}_x} \mbb{E}_{\mbb{P}} f(X_t)$ for uncertainty subsets $\mc{P}_x$, $x \in \mbb{R}^d$, satisfying \eqref{main-eq7} be such that \eqref{C4} and \eqref{C5} hold. If $f \in \mc{H}$ is such that $(t,x) \mapsto u(t,x):= T_t f(x)$ is continuous, then $u$ is a viscosity solution (in the sense of Definition~\ref{pre-21}) to \begin{align} \label{main-eq11} \begin{aligned}
	\partial_t u(t,x) - \sup_{\alpha \in I} &\bigg( b_{\alpha}(x) \cdot \nabla_x u(t,x) + \frac{1}{2} \tr\left(Q_{\alpha}(x) \cdot \nabla^2_x u(t,x)\right) \\
	& +\int_{y \neq 0} \big(u(t,x+y)-u(t,x)-\nabla_x u(t,x) \cdot h(y) \big) \, \nu_{\alpha}(x,dy) \bigg) = 0 \end{aligned}
	\end{align}
	with $u(0,x) = f(x)$.
\end{thm}

Theorem~\ref{main-3} generalizes \cite[Theorem 4.37]{julian} where the assertion was shown under stronger regularity assumptions -- \cite{julian} requires Lipschitz continuity of the uncertainty coefficients and of the mapping $f$ -- and under the additional assumption that \begin{equation}
	\sup_{\alpha \in I} \sup_{x \in \mbb{R}^d} \left( |b_{\alpha}(x)| + |Q_{\alpha}(x)| + \int_{y \neq 0} \min\{|y|,|y|^2\} \, \nu_{\alpha}(x,dy) \right)<\infty \label{main-eq12}.
\end{equation}
Clearly, \eqref{main-eq12} is more restrictive than the uniform boundedness condition \eqref{A2} since \eqref{main-eq12} poses an additional integrability assumption on the family of L\'evy measures at infinity; for instance, $\nu_{\alpha}(dy) = |y|^{-d-1} \, dy$ satisfies \eqref{A2} but not \eqref{main-eq12}. Since the uniform boundedness condition \eqref{A2} in Theorem~\ref{main-3} can be replaced by the uniform continuity condition \eqref{A1}, Theorem~\ref{main-3} even allows us to study equations with unbounded coefficients; this seems to be a novelty in the literature. We will take advantage of this when we study HJB equations associated with solutions to stochastic differential equations, see Corollary~\ref{main-11}. 

\begin{bem_thm} \label{main-5} \begin{enumerate}
	\item\label{main-5-i} Hollender \cite{julian} established a comparison principle for HJB equations  \eqref{main-eq11}, cf.\ \cite[Corollary 2.34]{julian}; it gives, in particular, a sufficient condition for the uniqueness of the solution to \eqref{main-eq11}.
	\item\label{main-5-ii} Theorem~\ref{main-3} requires continuity of the mapping $(t,x) \mapsto T_t f(x)$. Using the stochastic representation $T_t(x) = \sup_{\mbb{P} \in \mathfrak{P}_x} E_{{\mbb{P}}} f(X_t)$ and a maximal inequality, which we will derive in Section~\ref{max}, we will show that $T_t f(x)$ depends continuously on $t$ whenever $f \in \mc{H}$ is bounded and uniformly continuous, cf.\ Theorem~\ref{max-5}. The continuous dependence on the space variable $x$ is, however, in general hard to verify, see the discussion in \cite[Remark 4.43]{julian} for further details.
\end{enumerate} \end{bem_thm}

\section{Applications} \label{ex}

In this section we present applications of Theorem~\ref{main-3}. We are going to study existence and uniqueness results for Hamilton--Jacobi--Bellman (HJB) equations
\begin{align} \label{ex-eq3} \begin{aligned}
\partial_t u(t,x) - \sup_{\alpha \in I} &\bigg( b_{\alpha}(x) \cdot \nabla_x u(t,x) + \frac{1}{2} \tr(Q_{\alpha}(x) \cdot \nabla^2_x u(t,x)) \\
& +\int_{y \neq 0} \big(u(t,x+y)-u(t,x)-\nabla_x u(t,x) \cdot h(y) \big) \, \nu_{\alpha}(x,dy) \bigg) = 0 \end{aligned}
\end{align}
for the following particular cases: \begin{enumerate}
	\item The coefficients $(b_{\alpha}(x),Q_{\alpha}(x),\nu_{\alpha}(x,dy))$ are space homogeneous, i.\,e.\ do not depend on the space variable $x$. This leads to, so-called, L\'evy processes for sublinear expectations, cf.\ Proposition~\ref{main-6} and Corollary~\ref{main-7}. Our results apply, in particular, to classical L\'evy processes, cf.\ Corollary~\ref{main-8}.
	\item The coefficients $(b_{\alpha}(x),Q_{\alpha}(x),\nu_{\alpha}(x,dy))$ are of the form 
	\begin{equation*}
		b_{\alpha}(x) = \sigma(x) \hat{b}_{\alpha} \qquad Q_{\alpha}(x) = \sigma(x) \hat{Q}_{\alpha} \sigma(x)^T \qquad \nu_{\alpha}(x,dy) = \hat{\nu}_{\alpha} \circ \sigma(x)^{-1}
	\end{equation*}
	for a family of L\'evy triplets $(\hat{b}_{\alpha},\hat{Q}_{\alpha},\hat{\nu}_{\alpha})$, $\alpha \in I$, and a mapping $\sigma$. Such HJB equations are the evolution equations of solutions to stochastic differential equations driven by a sublinear L\'evy process, cf.\ Corollary~\ref{main-11}.
	\item The index set $I$ consists of a single element, i.\,e.\ \eqref{ex-eq3} becomes \begin{align*}  \begin{aligned}
	\partial_t u(t,x) - &\bigg( b(x) \cdot \nabla_x u(t,x) + \frac{1}{2} \tr(Q(x) \cdot \nabla^2_x u(t,x)) \\
	& +\int_{y \neq 0} \big(u(t,x+y)-u(t,x)-\nabla_x u(t,x) \cdot h(y) \big) \, \nu(x,dy) \bigg) = 0. \end{aligned}
	\end{align*}
	They appear as Kolmogorov backward equations of Feller processes on classical probability spaces, cf.\ Corollary~\ref{main-13}.
\end{enumerate}

As usual, $(X_t)_{t \geq 0}$ denotes the canonical process on the Skorohod space. Recall that a function $f: \mbb{R}^d \to \mbb{R}$ is called upper semi-analytic if the preimage $\{f>c\}=\{x \in \mbb{R}^d; f(x)>c\}$ is an analytic set for all $c \in \mbb{R}$, i.\,e.\  $\{f>c\}$ is the continuous image of a Polish space for all $c \in \mbb{R}$.

\begin{prop} \label{main-6}
	Let $(b_{\alpha},Q_{\alpha},\nu_{\alpha})$, $\alpha \in I$, be a family of L\'evy triplets. If \begin{equation*}
	\sup_{\alpha \in I} \left( |b_{\alpha}| + |Q_{\alpha}| + \int_{y \neq 0} \min\{1,|y|^2\} \, \nu_{\alpha}(dy) \right)< \infty, 
	\end{equation*}
	then the family of sublinear operators \begin{equation*}
	T_t f(x) := \mc{E}^x f(X_t) := \sup_{\mbb{P} \in \mathfrak{P}_x} \mbb{E}_{\mbb{P}} f(X_t), \qquad t \geq 0, \, x \in \mbb{R}^d,
	\end{equation*}
	with uncertainty subsets
	\begin{equation}
	\mathfrak{P}_x :=\left\{ \mbb{P} \in \mathfrak{P}_{\sem}^{\ac}(D_x); (b_s^{\mbb{P}},Q_s^{\mbb{P}},\nu_s^{\mbb{P}})(\omega) \in \bigcup_{\alpha \in I} \{(b_{\alpha},Q_{\alpha},\nu_{\alpha})\} \, \, \lambda(ds)\times \mbb{P}-\text{a.s.}\right\} \label{main-eq13}
	\end{equation}
	defines a sublinear Markov semigroup on each of the following spaces: \begin{enumerate}
		\item\label{main-6-i} the space of bounded upper semi-analytic functions,
		\item\label{main-6-ii} the space of bounded uniformly continuous functions,
		\item\label{main-6-iii} the space of bounded Lipschitz continuous functions.
	\end{enumerate}
	If the family of L\'evy measures $\nu_{\alpha}$, $\alpha \in I$, is tight at infinity in the sense that \begin{equation}
		\lim_{R \to \infty} \sup_{\alpha \in I} \int_{|y|>R} \, \nu_{\alpha}(dy)=0, \label{main-eq10}
	\end{equation}
	then $(T_t)_{t \geq 0}$ is a sublinear Markov semigroup on \begin{enumerate} \setcounter{enumi}{3}
		\item\label{main-6-iv} the space of bounded continuous functions.
	\end{enumerate}
\end{prop}

It can be shown that the semigroup $(T_t)_{t \geq 0}$ is spatially homogeneous on the space of bounded upper semi-analytic functions, i.\,e.\ \begin{equation}
T_t f(x) = \mc{E}^x f(X_t) = \mc{E}^0 f(x+X_t) = T_t (f(x+\cdot))(0),  \label{main-eq19}
\end{equation}
for any bounded upper semi-analytic function $f$, and that $(X_t)_{t \geq 0}$ has independent and stationary increments, see e.\,g.\ \cite[Remark 4.38]{julian} for details. Following \cite{denk,julian} we refer to the process $(X_t)_{t \geq 0}$ from Corollary~\ref{main-7} as \emph{L\'evy process for sublinear expectations with uncertainty coefficients $(b_{\alpha},Q_{\alpha},\nu_{\alpha})$}; we would like to mention that there is no standard terminology for this class of processes, for instance \cite{hu} calls them G-L\'evy processes.  \par \medskip

From Theorem~\ref{main-3} and Proposition~\ref{main-6} we obtain the following existence and uniqueness result for HJB equations \eqref{ex-eq3} with space homogeneous coefficients.

\begin{kor} \label{main-7}
	Let $h$ be a truncation function, and let $(b_{\alpha},Q_{\alpha},\nu_{\alpha})$, $\alpha \in I$, be a family of L\'evy triplets such that \begin{equation}
		\sup_{\alpha \in I} \left( |b_{\alpha}| + |Q_{\alpha}| + \int_{y \neq 0} \min\{1,|y|^2\} \, \nu_{\alpha}(dy) \right)< \infty. \label{main-eq14}
	\end{equation}
	Denote by $T_t f(x) = \mc{E}^x f(X_t)$ the family of sublinear operators introduced in Proposition~\ref{main-6}. The mapping $u(t,x) := T_t f(x)$ is a viscosity solution to \begin{align} \label{main-eq15}\begin{aligned}
		\partial_t u(t,x) - \sup_{\alpha \in I}& \bigg( b_{\alpha} \cdot \nabla_x u(t,x) + \frac{1}{2} \tr\left(Q_{\alpha} \cdot \nabla^2_x u(t,x)\right) \\
		&\qquad + \int_{y \neq 0} \big( u(t,x+y)-u(t,x)-\nabla_x u(t,x) \cdot h(y) \big) \, \nu_{\alpha}(dy) \bigg)=0 \end{aligned}
	\end{align}
	with $u(0,x)=f(x)$ in each of the following cases: \begin{enumerate}
		\item\label{main-7-i} $f$ is bounded and uniformly continuous,
		\item\label{main-7-ii} \eqref{main-eq10} holds and $f$ is bounded and continuous.
	\end{enumerate}
	If additionally the tightness condition \begin{equation}
		\lim_{r \to 0} \sup_{\alpha \in I} \int_{0<|y| \leq r} |y|^2 \, \nu_{\alpha}(dy) = 0 \quad \text{and} \quad \lim_{R \to \infty} \sup_{\alpha \in I} \int_{|y|>R} \, \nu_{\alpha}(dy)=0 \label{main-eq17}
	\end{equation}
	holds, then $u(t,x) = T_t f(x)$ is for any $f \in C_b(\mbb{R}^d)$ the unique viscosity solution to \eqref{main-eq15} with $u(0,x)=f(x)$.
\end{kor}

Corollary~\ref{main-7} generalizes \cite{nutz} where it was shown that $u(t,x) = \sup_{\mbb{P} \in \mathfrak{P}_x} \mbb{E}_{\mbb{P}} f(X_t)$ is a viscosity solution to \eqref{main-eq15} under the additional assumptions that $f$ is Lipschitz continuous and that $\sup_{\alpha \in I} \int_{|y|>1} |y| \, \nu_{\alpha}(dy)<\infty$. Recently, Denk et al.\ \cite{denk} studied semigroups associated with L\'evy processes for sublinear expectations and showed that \eqref{main-eq14} implies the existence of a viscosity solution to \eqref{main-eq15} for initial values $u(0,x)=f(x)$ which are bounded and uniformly continuous; this corresponds to Corollary~\ref{main-7}\eqref{main-7-i}. Let us mention that Denk et al.\ do not have the representation $u(t,x) = \sup_{\mbb{P} \in \mathfrak{P}_x} \mbb{E}_{\mbb{P}} f(X_t)$ for the solution; this stochastic representation is useful to derive further information on the solution, e.\,g.\ to study regularity and growth properties. Moreover, it allows us to interpret L\'evy processes for sublinear expectations as generalizations of classical L\'evy processes under uncertainty of the semimartingale characteristics. \par \medskip
For the particular case that the index set $I$ consists of a finitely many elements, the tightness condition \eqref{main-eq10} is automatically satisfied, and therefore Corollary~\ref{main-7} gives the following result for classical L\'evy processes.

\begin{kor} \label{main-8}
	Let $(X_t)_{t \geq 0}$ be a $d$-dimensional L\'evy process on a classical probability space $(\Omega,\mc{A},\mbb{P})$, and denote by \begin{equation*}
		\psi(\xi) = -ib \cdot \xi + \frac{1}{2} \xi \cdot Q \xi + \int_{y \neq 0} \left( 1-e^{iy \cdot \xi} + i \xi \cdot h(y) \right) \, \nu(dy), \quad \xi \in \mbb{R}^d,
	\end{equation*}
	its characteristic exponent. If $f : \mbb{R}^d \to \mbb{R}$ is continuous and bounded, then $T_t f(x) := \mbb{E}f(x+X_t)$ is the unique viscosity solution to \begin{align*} \begin{aligned}
			\partial_t u(t,x) -& \bigg( b \cdot \nabla_x u(t,x) + \frac{1}{2} \tr\left(Q \cdot \nabla^2_x u(t,x)\right) \\
			&\qquad + \int_{y \neq 0} \big( u(t,x+y)-u(t,x)-\nabla_x u(t,x) \cdot h(y) \big) \, \nu(dy) \bigg)=0. \end{aligned}
		\end{align*}
	with $u(0,x) = f(x)$.
\end{kor}

Next we investigate HJB equations \eqref{ex-eq3} with coefficients $(b_{\alpha}(x),Q_{\alpha}(x),\nu_{\alpha}(x,dy))$ of the form \begin{equation*}
	b_{\alpha}(x) = \sigma(x) \hat{b}_{\alpha} \qquad Q_{\alpha}(x) = \sigma(x) \hat{Q}_{\alpha} \sigma(x)^T \qquad \nu_{\alpha}(x,dy) = \hat{\nu}_{\alpha} \circ \sigma(x)^{-1}
\end{equation*}
for a family of L\'evy triplets $(\hat{b}_{\alpha},\hat{Q}_{\alpha},\hat{\nu}_{\alpha})$, $\alpha \in I$, and a mapping $\sigma$; here $\hat{\nu}_{\alpha} \circ \sigma(x)^{-1}$ denotes the pullback of the measure $\hat{\nu}_{\alpha}$ under $\sigma(x)$. It turns out that such HJB equations are the evolution equations of solutions to stochastic differential equations driven by a sublinear L\'evy process. Before we can state the result, we need to introduce stochastic integrals on sublinear expectation spaces. It is well-known, see e.\,g.\ Jacod \cite{jacod}, that for any semimartingale $(X_t)_{t \geq 0}$ on a classical probability space $(\Omega,\mc{A},\mbb{P})$ and any locally bounded predictable process $(H_t)_{t \geq 0}$ the stochastic integral $\int_0^t H_s \, dX_s$ exists. We will use the notation \begin{equation*}
	I_{\mbb{P}}(H,X)_t := \int_0^t H_s \, dX_s
\end{equation*}
to emphasize that the stochastic integral depends on the underlying probability measure $\mbb{P}$. The next definition is adapted from \cite[Definition 4.51]{julian}.

\begin{defn} \label{main-10} 
	Let $(X_t)_{t \geq 0}$ be a $d$-dimensional $(\mathcal{F}_t)_{t \geq 0}$-adapted stochastic process on a measurable space $(\Omega,\mc{A})$, and let $\mathfrak{P} \subseteq \mathfrak{P}_{\sem}^{\ac}(\Omega)$, i.\,e.\ a family of probability measures $\mbb{P}$ on $(\Omega,\mc{A})$ such that $(X_t,\mc{F}_t)_{t \geq 0}$ is a semimartingale (with respect to $\mbb{P}$) which has absolutely continuous differential characteristics with respect to Lebesgue measure. For a locally bounded predictable process $(H_t)_{t \geq 0}$ the \emph{stochastic integral} of $(H_t)_{t \geq 0}$ with respect to $(X_t)_{t \geq 0}$, \begin{equation*}
		I_{\mathfrak{P}}(H,X) = (I_{\mathfrak{P}}(H,X)_t)_{t \geq 0}: \Omega \to D([0,\infty))
	\end{equation*}
	is defined as the stochastic process adapted to the $\mathfrak{P}$-universally augmented filtration $(\mc{F}_{t+}^{\mathfrak{P}})_{t \geq 0}$, \begin{equation*}
		\mc{F}_{t+}^{\mathfrak{P}} := \bigcap_{\mbb{P} \in \mathfrak{P}} \mc{F}_{t+}^{\mbb{P}}
	\end{equation*}
	with \cadlag sample paths such that \begin{equation*}
		I_{\mathfrak{P}}(H,X) = I_{\mbb{P}}(H,X)
	\end{equation*}
	holds $\mbb{P}$-almost surely for any $\mbb{P} \in \mathfrak{P}$.
\end{defn}

Let us mention that the stochastic integral can be defined for a larger space of integrands; in this paper we restrict ourselves to locally bounded processes. The following existence and uniqueness result for stochastic differential equations on sublinear expectation spaces is a direct consequence of \cite[Theorem 4.57]{julian}.

\begin{thm} \label{main-105}
	Let $(X_t)_{t \geq 0}$ and $\mathfrak{P}$ be as in Definition~\ref{main-10}. If $\sigma: \mbb{R}^k \to \mbb{R}^{k \times d}$ is a Lipschitz continuous function, then the stochastic differential equation \begin{equation}
		Z =x+ I_{\mathfrak{P}}(\sigma(Z_{-\cdot}),X) \label{main-eq23}
	\end{equation}
	has for any initial point $Z_0 = x \in \mbb{R}^k$ a solution $(Z_t)_{t \geq 0}$ with \cadlag sample paths which is adapted to the universally completed filtration \begin{equation*}
		\mc{F}_t^{\mathfrak{P}} := \bigcap_{\mbb{P} \in \mathfrak{P}} \mc{F}_t^{\mbb{P}} := \bigcap_{\mbb{P} \in \mathfrak{P}} \{A \cup N; A \in \mc{F}_t, \exists B \in \mc{A}, B \supseteq N: \mbb{P}(B)=0\}.
	\end{equation*}
	The solution is unique in the following sense: If $(Z_t')_{t \geq 0}$ is another $\mc{F}_t^{\mathfrak{P}}$-adapted solution with \cadlag sample paths, then $Z=Z'$ $\mbb{P}$-a.s. for any $\mbb{P} \in \mathfrak{P}$.
\end{thm}

Note that the solution $(Z_t)_{t \geq 0}$ to \eqref{main-eq23} satisfies by the very definition of the stochastic integral, cf.\ Definition~\ref{main-10}, \begin{equation*}
	Z_t = x + I_{\mbb{P}}(\sigma(Z_{-}),X)_t = x+ \int_0^t \sigma(Z_{s-}) \, dX_s
\end{equation*}
$\mbb{P}$-almost surely for any $\mbb{P} \in \mathfrak{P}$. \par \medskip

Theorem~\ref{main-3} allows us to study the evolution equations of solutions to stochastic differential equations driven by a L\'evy process for sublinear expectations. 
We obtain the following existence and uniqueness result which generalizes \cite[Remark 4.62]{julian}.

\begin{kor} \label{main-11}
	Let $(X_t)_{t \geq 0}$ be the $d$-dimensional sublinear L\'evy process from Corollary~\ref{main-7} with uniformly bounded uncertainty coefficients $(b_{\alpha},Q_{\alpha},\nu_{\alpha})$, $\alpha \in I$, and truncation function $h$.  Assume that the uncertainty coefficients satisfy the tightness condition \eqref{main-eq17} and that $\bigcup_{\alpha \in I} \{(b_{\alpha},Q_{\alpha},\nu_{\alpha})\}$ is convex and closed (see Remark~\ref{main-12} below). For a Lipschitz continuous function $\sigma: \mbb{R}^k \to \mbb{R}^{k \times d}$ which grows at most sublinearly let $(Z_t^x)_{t \geq 0}$ be the unique solution to the SDE \begin{equation*}
		Z^x = x +I_{\mathfrak{P}^0}(\sigma(Z_{\cdot-}^x),X), 
	\end{equation*}
	cf.\ Theorem~\ref{main-105}, for the uncertainty subset $\mathfrak{P}^0$ defined in \eqref{main-eq13}. If $f: \mbb{R}^k \to \mbb{R}$ is a continuous bounded function, then \begin{equation*}
		u(t,x) := P_t f(x) := \mc{E}^0 f(Z_t^x) := \sup_{\mbb{P} \in \mathfrak{P}^0} \mbb{E}_{\mbb{P}} f(Z_t^x), \qquad t \geq 0, \, x \in \mbb{R}^k,
	\end{equation*}
	is the unique viscosity solution to \begin{align} 
		\label{main-eq25} \begin{aligned}
			\partial_t u(t,x) &- \sup_{\alpha \in I} \bigg( \sigma(x) b_{\alpha} \nabla_x u(t,x) + \frac{1}{2} \tr\left(\sigma(x) Q_{\alpha} \sigma(x)^T \nabla^2_x u(t,x)\right) \\
			&+ \int_{y \neq 0} \big( u(t,x+\sigma(x) \cdot y)-u(t,x) - \nabla_x u(t,x) \sigma(x) h(y) \big) \, \nu_{\alpha}(dy) \bigg) = 0
		\end{aligned}
	\end{align}
	with $u(0,x)=f(x)$.
\end{kor}

\begin{bem_thm} \label{main-12}
	In Corollary~\ref{main-11} we assume that the set $ \{(b_{\alpha},Q_{\alpha},\nu_{\alpha}); \alpha \in I\}$ is closed; let us explain which topology we consider. For each $\alpha \in I$ the L\'evy triplet $(b_{\alpha},Q_{\alpha},\nu_{\alpha})$ is an element of the cartesian product \begin{equation*}
		\Pi := \mbb{R}^d \times \mbb{S}^{d \times d}_+ \times \mathfrak{L}(\mbb{R}^d)
	\end{equation*}
	where $\mbb{S}^{d \times d}_+$ is the space of symmetric semi-positive definite matrices $Q \in \mbb{R}^{d \times d}$ and $\mathfrak{L}(\mbb{R}^d)$ is the family of L\'evy measures on $\mbb{R}^d \backslash \{0\}$. We consider $\Pi$ endowed with the product topology; as usual, $\mbb{R}^d$ and $\mbb{S}^{d \times d}_+$ are endowed with the Euclidean metric, and the topology on $\mathfrak{L}(\mbb{R}^d)$ is induced by the Wasserstein metric \begin{equation*}
		d_W(\mu,\nu) := \sup_{f}  \left| \int_{\mbb{R}^d} f(x) \, \min\{|x|^2,1\} \, \mu(dx) - \int_{\mbb{R}^d} f(x) \, \min\{|x|^2,1\} \, \nu(dx) \right|
	\end{equation*}	
	where the supremum is taken over all functions $f:\mbb{R}^d \to \mbb{R}$ which satisfy \begin{equation*}
		\sup_{x \in \mbb{R}^d} |f(x)| \leq 1 \qquad \text{and} \qquad \sup_{x \neq y} \frac{|f(x)-f(y)|}{|x-y|} \leq 1.
	\end{equation*}
\end{bem_thm}

We close this section with a result on HJB equations associated with Feller processes on classical probability spaces. Recall that a Markov process $(X_t)_{t \geq 0}$ with \cadlag sample paths and semigroup $T_t f(x) = \mbb{E}^x f(X_t)$ is a Feller process if $(T_t)_{t \geq 0}$ is strongly continuous on $C_{\infty}(\mbb{R}^d)$, the space of continuous functions vanishing at infinity, i.\,e.\ \begin{equation*}
	\forall f \in C_{\infty}(\mbb{R}^d): \quad \lim_{t \to 0} \|T_t f-f\|_{\infty} = 0,
\end{equation*}
and each $T_t$ has the Feller property, i.\,e.\ $T_t(C_{\infty}(\mbb{R}^d)) \subseteq C_{\infty}(\mbb{R}^d)$. In particular, $(T_t)_{t \geq 0}$ is a sublinear Markov semigroup on $\mc{H} := C_{\infty}(\mbb{R}^d)$ in the sense of Definition~\ref{main-0}. The sublinear infinitesimal generator associated with this sublinear Markov semigroup coincides with the classical infinitesimal generator associated with the Feller process $(X_t)_{t \geq 0}$; this follows from the maximal dissipativity of the (classical) infinitesimal generator, see e.\,g.\ \cite[Theorem 1.33]{ltp}. If $(X_t)_{t \geq 0}$ is a Feller process with infinitesimal generator $(A,\mc{D}(A))$ which satisfies $C_c^{\infty}(\mbb{R}^d) \subseteq \mc{D}(A)$, then $A$ equals, when restricted to $C_c^{\infty}(\mbb{R}^d)$, a pseudo-differential operator with negative definite symbol $q$, cf.\ \eqref{pseudo} and \eqref{sym}. For a detailed discussion of Feller processes and their infinitesimal generators we refer to the monographs \cite{ltp,jac2}.

\begin{kor} \label{main-13}
	Let $(X_t)_{t \geq 0}$ be a $d$-dimensional Feller process with semigroup $(T_t)_{t \geq 0}$ and infinitesimal generator $(A,\mc{D}(A))$ such that $C_c^{\infty}(\mbb{R}^d) \subseteq \mc{D}(A)$. Assume that characteristics $(b,Q,\nu)$ of the associated symbol $q$ depends uniform continuously on $x$, in the sense that $x \mapsto b(x)$, $x \mapsto Q(x)$ are uniformly continuous and that $x \mapsto \int_{y \neq 0} g(y) \, \nu(x,dy)$ is uniformly continuous for any $g \in C_b^1(\mbb{R}^d)$ such that $|g(y)| \leq \min\{1,|y|^2\}$, $y \in \mbb{R}^d$.	The mapping $u(t,x) := T_t f(x)$ is a viscosity solution to \begin{align} \label{main-eq31} \begin{aligned}
		\partial_t u(t,x)  - \bigg( b(x) \cdot \nabla_x u(t,x) &+ \frac{1}{2} \tr\big(Q(x) \nabla^2_x u(t,x)\big)  \\
		&+ \int_{y \neq 0} \big( u(t,x+y)-u(t,x) - \nabla_x u(t,x) \cdot h(y) \big) \, \nu(x,dy) \bigg)= 0 \end{aligned}
	\end{align}
	with $u(0,x)=f(x)$ in each of the following cases: \begin{enumerate}
		\item\label{main-13-i} $f \in C_{\infty}(\mbb{R}^d)$ and $q$ has bounded coefficients, i.\,e.\ \begin{equation*}
				\sup_{x \in \mbb{R}^d} \left( |b(x)|+|Q(x)| + \int_{y \neq 0} \min\{1,|y|^2\} \, \nu(x,dy) \right) < \infty.
			\end{equation*}
		\item\label{main-13-ii} $f \in C_b(\mbb{R}^d)$ and $q$ satisfies the uniform continuity condition \begin{equation}
				\lim_{r \to \infty} \sup_{|\xi| \leq r^{-1}} \sup_{|x| \leq r} |q(x,\xi)| = 0. \label{main-eq29}
			\end{equation}
		\end{enumerate}
\end{kor}

\begin{bem_thm} \label{main-15} \begin{enumerate}
	\item  If \eqref{main-eq29} holds then the continuity assumptions on the coefficients can be relaxed; it suffices to assume that the coefficients depend continuously on $x$, i.\,e.\ that the mappings $x \mapsto b(x)$, $x \mapsto Q(x)$ and $x \mapsto \int g(y) \, \nu(x,dy)$ are continuous. This is a direct consequence of Theorem~\ref{main-3} and the proof of Corollary~\ref{main-13}.
	\item Corollary~\ref{main-13} applies, in particular, if $(X_t)_{t \geq 0}$ is a L\'evy process, see also Corollary~\ref{main-8}.
\end{enumerate}
\end{bem_thm}

\section{Maximal inequality for sublinear expectations} \label{max}

Let $(X_t)_{t \geq 0}$ be the canonical process on the Skorohod space. In the first part of this section we establish a maximal inequality of the form \begin{equation*}
	\mc{P}^x \left( \sup_{s \leq t} |X_s-x|>r \right) \leq c_r t, \qquad t \geq 0,
\end{equation*}
for sublinear expectations $\mc{P}^x = \sup_{\mbb{P} \in \mathfrak{P}_x} \mbb{P}$ with uncertainty subset $\mathfrak{P}_x$,\begin{equation*}
	\mathfrak{P}_x \subseteq \left\{ \mbb{P} \in \mathfrak{P}_{\sem}^{\ac}(D_x); (b_s^{\mbb{P}},Q_s^{\mbb{P}},\nu_s^{\mbb{P}})(\omega) \in \bigcup_{\alpha \in I} \{(b_{\alpha},Q_{\alpha},\nu_{\alpha})(X_s(\omega))\} \, \, \lambda(ds)\times \mbb{P}\text{-a.s.}\right\}.
\end{equation*}
The idea of the proof goes back to Schilling \cite{rs-growth} who obtained a maximal inequality for Feller processes, see also \cite[Theorem 5.1]{ltp}. The maximal inequality has turned out to be a very useful tool to study distributional and path properties of Feller processes, cf.\ \cite{ltp}. Recently, a localized version of the maximal inequality was derived in \cite{ihke} to study domains of Feller generators, and in \cite{perpetual} the maximal inequality was used in the context of martingale problems to give a sufficient condition for the non-explosion of solutions. Since the proof of the maximal inequality for Feller processes relies essentially on Dynkin's formula, we can extend it to our framework.

\begin{prop}[Maximal inequality] \label{max-1}
	Let $(b_{\alpha}(x),Q_{\alpha}(x),\nu_{\alpha}(x,\cdot))$, $x \in \mbb{R}^d$, $\alpha \in I$, be a family of L\'evy triplets which is uniformly bounded on compact sets, i.\,e.\ which satisfies \eqref{main-eq4}. For a given truncation function $h$ denote by \begin{equation*}
		q_{\alpha}(x,\xi) = -ib_{\alpha}(x) \cdot \xi  + \frac{1}{2} \xi \cdot Q_{\alpha}(x) \xi + \int_{y \neq 0} (1-e^{iy \cdot \xi}+i \xi \cdot h(y)) \, \nu_{\alpha}(x,dy)
	\end{equation*}
	the associated family of continuous negative definite symbols. For any uncertainty subset $\mathfrak{P}_x$ with \begin{equation}
		\mathfrak{P}_x \subseteq \left\{ \mbb{P} \in \mathfrak{P}_{\sem}^{\ac}(D_x); (b_s^{\mbb{P}},Q_s^{\mbb{P}},\nu_s^{\mbb{P}})(\omega) \in \bigcup_{\alpha \in I} \{(b_{\alpha},Q_{\alpha},\nu_{\alpha})(X_s(\omega))\} \, \, \lambda(ds)\times \mbb{P}\text{-a.s.}\right\} \label{max-eq13}
	\end{equation}
	there exists  $c>0$ such that \begin{equation}
		\sup_{\mbb{P} \in \mathfrak{P}_x} \mbb{P} \left( \sup_{s \leq t} |X_s-x|> r \right) \leq ct \sup_{\alpha \in I} \sup_{|z-x| \leq r} \sup_{|\xi| \leq r^{-1}} |q_{\alpha}(z,\xi)| \fa t \geq 0, \, r>0. \label{max-eq15}
	\end{equation}
	The constant $c>0$ can be chosen independently from the starting point $x \in \mbb{R}^d$ and the family $(b_{\alpha}(x),Q_{\alpha}(x),\nu_{\alpha}(x,\cdot))$, $\alpha \in I$, $x \in \mbb{R}^d$.
\end{prop}

\begin{proof}
	Fix $\mbb{P} \in \mathfrak{P}_x$ and $u \in C_c^{\infty}(\mbb{R}^d)$ such that $\spt u \subseteq B(0,1)$ and $0 \leq u \leq 1 = u(0)$. If we set  $u_r^x(\cdot) := u((\cdot-x)/r)$, then Dynkin's formula, cf.\ Lemma~\ref{app-1}, shows that
	\begin{align*}
		\mbb{E}_{\mbb{P}} u_r^x(X_{\tau_r^x \wedge t})-1
		= \mbb{E}_{\mbb{P}} \left( \int_{(0,t \wedge \tau_r^x)} A_s^{\mbb{P}} u_r^x (X_{s-}) \, ds \right)
	\end{align*}
	where $\tau_r^x$ denotes the first exit time from the closed ball $\overline{B(x,r)}$ and \begin{align*}
		A_s^{\mbb{P}} f(z) := b_s^{\mbb{P}} \cdot \nabla f(z) + \frac{1}{2} \tr\left(Q_s^{\mbb{P}} \cdot \nabla^2 f(z)\right) + \int_{y \neq 0} \big(f(z+y)-f(z)-\nabla f(z) \cdot h(y) \big) \, \nu_s^{\mbb{P}}(dy). 
	\end{align*}
	Thus, \begin{align*}
		\mbb{P}(\tau_r^x \leq t)
		\leq \mbb{E}_{\mbb{P}} (1-u_r^x(X_{t \wedge \tau_r^x}))
		&= - \mbb{E}_{\mbb{P}} \left( \int_{(0,t \wedge \tau_r^x)} A_s^{\mbb{P}} u_r^x (X_{s-}) \, ds \right).
	\end{align*}
	Because of the structural assumption \eqref{max-eq13} this implies that \begin{align*}
		\mbb{P}(\tau_r^x \leq t)
		\leq \sup_{\alpha \in I} \left| \mbb{E}_{\mbb{P}} \left( \int_{(0,t \wedge \tau_r^x)} q_{\alpha}({X_{s-}},D) u_r^x(X_{s-}) \, ds \right) \right|;
	\end{align*}
	here $q_{\alpha}(z,D)$ denotes the pseudo-differential operator with symbol $q_{\alpha}$, cf.\ \eqref{pseudo}. Hence, \begin{align*}
		\mbb{P}(\tau_r^x \leq t)
		\leq t \sup_{\alpha \in I} \sup_{|z-x| \leq r} |q_{\alpha}(z,D) u_r^x(z)|.
	\end{align*}
	Since the Fourier transform $\hat{u}_r^x$ satisfies $|\hat{u}_r^x(\xi)| = r^d |\hat{u}(r \xi)|$ a change of variables gives \begin{align*}
		|q_{\alpha}(z,D) u_r^x(z)| 
		&= \left| \int_{\mbb{R}^d} q_{\alpha}(z,\xi) e^{iz \cdot \xi} \hat{u}_r^x(\xi) \, d\xi \right|
		\leq \int_{\mbb{R}^d} |q_{\alpha}(z,r^{-1} \xi)| \, |\hat{u}(\xi)| \, d\xi.
	\end{align*}
	Using that $|q_{\alpha}(z,\eta)| \leq 2 (1+|\eta|^2)\sup_{|\xi| \leq 1} |q_{\alpha}(z,\xi)|$, see e.\,g.\ \cite[Theorem 2.31]{ltp}, we conclude that \begin{equation*}
		\mbb{P}(\tau_r^x \leq t)
		\leq ct \sup_{\alpha \in I} \sup_{|z-x| \leq r} \sup_{|\xi| \leq r^{-1}} |q_{\alpha}(z,\xi)|
	\end{equation*}
	for $c:= 2 \int_{\mbb{R}^d} (1+|\xi|^2) |\hat{u}(\xi)| \, d\xi$, cf.\ \cite[Theorem 5.1]{ltp} for more details. Taking the supremum over $\mbb{P} \in \mathfrak{P}_x$, this proves the assertion.
\end{proof}

\begin{bem_thm} \label{max-2}
	If $K \subseteq \mbb{R}^d$ is such that \begin{equation*}
		\sup_{\alpha \in I} \sup_{x \in K} \left( |b_{\alpha}(x)| + |Q_{\alpha}(x)| + \int_{y \neq 0} \min\{|y|^2,1\} \, \nu_{\alpha}(x,dy) \right)<\infty,
	\end{equation*}
	then \begin{equation*}
		\sup_{\alpha \in I} \sup_{x \in K} \sup_{|\xi| \leq r^{-1}} |q_{\alpha}(x,\xi)| < \infty
	\end{equation*}
	for any $r>0$; this follows easily from the fact that we can find for any truncation function $h$ a constant $C>0$ such that \begin{equation*}
		|1-e^{iy \cdot \xi} + ih(y) \cdot \xi| \leq C \min\{1,|y|^2 |\xi|^2\} \fa y,\xi \in \mbb{R}^d.
	\end{equation*}
	In particular, the boundedness condition \eqref{main-eq4} ensures that the supremum on the right-hand side of \eqref{max-eq15} is finite.
\end{bem_thm}

The maximal inequality allows us to study the regularity of the mapping $t \mapsto T_t f(x)$ for sublinear Markov semigroups $(T_t)_{t \geq 0}$. 

\begin{thm}[Continuity in time]\label{max-5}
	Let $q_{\alpha}(x,\cdot): \mbb{R}^d \to \mbb{C}$, $x \in \mbb{R}^d$, $\alpha \in I$, be a family of continuous negative definite functions with characteristics $(b_{\alpha}(x),Q_{\alpha}(x),\nu_{\alpha}(x,\cdot))$, $\alpha \in I$, $x \in \mbb{R}^d$, which is uniformly bounded on compact sets, i.\,e.\ which satisfies \eqref{main-eq4}. Let \begin{equation*}
		\mathfrak{P}_x \subseteq \left\{ \mbb{P} \in \mathfrak{P}_{\sem}^{\ac}(D_x); (b_s^{\mbb{P}},Q_s^{\mbb{P}},\nu_s^{\mbb{P}})(\omega) \in \bigcup_{\alpha \in I} \{(b_{\alpha},Q_{\alpha},\nu_{\alpha})(X_s(\omega))\} \, \, \lambda(ds)\times \mbb{P}-\text{a.s.}\right\}
	\end{equation*}
	be such that  \begin{equation*}
		T_t f(x) := \mc{E}^x f(X_t) :=\sup_{\mbb{P} \in \mathfrak{P}_x} \mbb{E}_{\mbb{P}}f(X_t), \qquad t \geq 0, \, x \in \mbb{R}^k.
	\end{equation*}
	defines a sublinear Markov semigroup on a convex cone $\mc{H}$ of real-valued functions. \begin{enumerate}
		 \item\label{max-5-ii} If $M:=\sup_{n \in \mbb{N}} M_n < \infty$ and $f \in \mc{H}$ is bounded and uniformly continuous, then $t \mapsto T_t f(x)$ is continuous uniformly in $x \in \mbb{R}^d$.
		\item\label{max-5-iii} If $f \in \mc{H}$ is bounded and continuous and \begin{equation*} \lim_{r \to \infty} \sup_{|x| \leq 2r} \sup_{|\xi| \leq r^{-1}} |q_{\alpha}(x,\xi)| = 0 \end{equation*} then $t \mapsto T_t f(x)$ is continuous uniformly in $x \in K$ for any compact set $K \subseteq \mbb{R}^d$.
	\end{enumerate}
\end{thm}

Let us mention that a sufficient condition for the family $(T_t)_{t \geq 0}$ to be a sublinear Markov semigroup was established by Hollender \cite[Remark 4.33]{julian}.

\begin{proof}[Proof of Theorem~\ref{max-5}]
	Let $f \in \mc{H}$ be a bounded function. Because of the subadditivity of $T_s$ we have \begin{equation*}
		T_s T_{t-s} f = T_s (T_{t-s}f-f+f) \leq T_s(T_{t-s} f- f) + T_s f
	\end{equation*}
	for any $s \leq t$ and by combining this with the Markov property we find that  \begin{align*}
		T_t f(x)-T_s f(x) = T_{s} T_{t-s} f(x) - T_s f(x) \leq T_s(T_{t-s}f-f) = \mc{E}^x \left( E^{X_s} f(X_{t-s}) - f(X_s) \right)
	\end{align*}
	for any $x \in \mbb{R}^d$ and $s \leq t$. In exactly the same fashion we obtain that \begin{align*}
		T_s f(x)-T_t f(x) \leq \mc{E}^x  \left( f(X_{s}) - E^{X_s} f(X_{t-s}) \right).
	\end{align*}
	Interchanging the roles of $s$ and $t$ we conclude that \begin{align}
		|T_t f(x)-T_s f(x)| 
		\leq \mc{E}^x \left( |E^z f(X_{|t-s|}) - f(z)| \bigg|_{z=X_{\min\{s,t\}}} \right)
		\leq \sup_{z \in \mbb{R}^d} E^z |f(X_{|t-s|})-f(z)| \label{max-eq21}
	\end{align}
	for any $s,t \geq 0$ and $x \in \mbb{R}^d$. Applying the maximal inequality \eqref{max-eq15} we get \begin{align*}
		|T_t f(x)-T_s f(x)|
		&\leq \sup_{|u-v| \leq \delta} |f(u)-f(v)| + 2\|f\|_{\infty} \sup_{z \in \mbb{R}^d} P^z \left( \sup_{r \leq |t-s|} |X_r-z|>\delta \right) \\
		&\leq  \sup_{|u-v| \leq \delta} |f(u)-f(v)| + 2c \|f\|_{\infty} |t-s| \sup_{\alpha \in I} \sup_{z \in \mbb{R}^d} \sup_{|\xi| \leq \delta^{-1}} |q_{\alpha}(z,\xi)|
	\end{align*}
	for any $\delta>0$. Since the boundedness of the coefficients implies that the second term on the right-hand side is finite, cf.\ Remark~\ref{max-2}, we infer that $t \mapsto T_t f(x)$ is continuous uniformly in $x \in \mbb{R}^d$ for any bounded function $f \in \mc{H}$ which is uniformly continuous, and this proves \eqref{max-5-ii}. To prove \eqref{max-5-iii} we fix $f \in \mc{H} \cap C_b(\mbb{R}^d)$ and note that, by \eqref{max-eq21}, \begin{align*}
		|T_t f(x)-T_s f(x)| \leq \sup_{|z| \leq R} E^z(|f(X_{|t-s|})-f(z)|) + 2 \|f\|_{\infty} \mc{P}^x \left( \sup_{r \leq \min\{s,t\}} |X_r-x|>R \right)
	\end{align*}
	implying \begin{align*}
		&|T_tf(x)-T_s f(x)| \\
		&\leq \sup_{\substack{|u|,|v| \leq R+\delta \\ |u-v| \leq \delta}} |f(u)-f(v)| + 2\|f\|_{\infty} \sup_{|z| \leq R} P^z \left( \sup_{r \leq |t-s|} |X_r-z|>\delta \right) + 2 \|f\|_{\infty} \mc{P}^x \left( \sup_{r \leq \min\{s,t\}} |X_r-x|>R \right)
	\end{align*}
	for any $\delta>0$. Applying once more the maximal inequality we find that \begin{align*}
		&|T_t f(x)-T_s f(x)| \\
		&\quad \leq \sup_{\substack{|u|,|v| \leq R+\delta \\ |u-v| \leq \delta}} |f(u)-f(v)|+ c |t-s| \sup_{|z| \leq R+\delta} \sup_{|\xi| \leq \delta^{-1}} \sup_{\alpha \in I} |q_{\alpha}(z,\xi)| + c (t+s) \sup_{|z-x| \leq R} \sup_{|\xi| \leq R^{-1}} \sup_{\alpha \in I} |q_{\alpha}(z,\xi)|
	\end{align*}
	for some absolute constant $c>0$. If $K \subseteq \mbb{R}^d$ is a compact set and $\eps>0$ fixed, then the growth assumption on $q_{\alpha}$ entails that we can choose $R>0$ sufficiently large such that the third term on the right hand side is less than $\eps$ for $x \in K$. The uniform continuity of $f$ on compacts implies that the first term on the right-hand side is less than $\eps$ for small $\delta>0$ and finally the second term will be small for $|t-s|<\varrho$ for $\varrho=\varrho(R,\delta)>0$ small enough. This proves the assertion.
\end{proof}

\section{Proofs} \label{p}

In this section we prove the results which we stated in Section~\ref{main} and \ref{ex}. The following lemma is the key ingredient for the proof of Theorem~\ref{main-3}. As usual, we denote by $(X_t)_{t \geq 0}$ the canonical process on the Skorohod space.

\begin{lem} \label{p-1}
	Let $(b_{\alpha}(x),Q_{\alpha}(x),\nu_{\alpha}(x,\cdot))$ and $\mc{E}^x = \sup_{\mbb{P} \in \mathfrak{P}_x} \mbb{E}_{\mbb{P}}$ be as in Theorem~\ref{main-3}. If $f \in C^2_b(\mbb{R}^d)$ is such that $f$ and its derivatives up to order $2$ are uniformly continuous, then \begin{equation*}
		\lim_{t \to 0} \frac{\mc{E}^x f(X_t)-f(x)}{t} = \sup_{\alpha \in I} A^{\alpha} f(x) \fa x \in \mbb{R}^d;
	\end{equation*}
	here $A^{\alpha}f(x) := q_{\alpha}(x,D)f(x)$ denotes the pseudo-differential operator with symbol $q_{\alpha}$, i.\,e.\ \begin{align*}
		A^{\alpha}f(x) &= b_{\alpha}(x) \cdot \nabla f(x) + \frac{1}{2} \tr\left(Q_{\alpha}(x) \nabla^2 f(x)\right) \\ &\quad + \int_{y \neq 0} \big( f(x+y)-f(x)-\nabla f(x) \cdot h(y) \big) \, \nu_{\alpha}(x,dy).
	\end{align*}
\end{lem}

The idea of the proof is similar to \cite[Proof of Theorem 4.37]{julian} but we use the maximal inequality, Proposition~\ref{max-1}, to avoid additional assumptions on the existence of moments.  

\begin{proof}[Proof of Lemma~\ref{p-1}]
	For simplicity of notation we give the proof only in dimension $d=1$. First we are going to prove the assertion under  \eqref{A1}; at the very end of the proof we will discuss how the proof has to be modified to obtain the assertion under the uniform boundedness condition \eqref{A2}. We divide the proof into several steps. \par
	\textbf{Step 1: Show that}  \begin{equation}
			\lim_{t \to 0} \frac{\mc{E}^x f(X_{t \wedge \tau_r^x})-f(x)}{t} = \lim_{t \to 0} \sup_{\mbb{P} \in \mathfrak{P}_x} \frac{1}{t} \mbb{E}_{\mbb{P}} \left( \int_{(0,t \wedge \tau_r^x)} A_s^{\mbb{P}} f(x) \, ds \right) \label{p-eq11}
		\end{equation}
		for any $r>0$ where $\tau_r^x$ denotes the first exit time of the canonical process $(X_t)_{t \geq 0}$ from the closed ball $\overline{B(x,r)}$ and
		\begin{equation*}
			A_s^{\mbb{P}} f(z) := b_s^{\mbb{P}} f'(z) + \frac{1}{2} Q_s^{\mbb{P}}f''(z) + \int_{y \neq 0} \big( f(z+y)-f(z)-f'(z) h(y)\big)  \, \nu_s^{\mbb{P}}(dy). 
		\end{equation*}
		Proof of \eqref{p-eq11}: Since $f$ and its derivatives up to order $2$ are uniformly continuous, we can choose for any $\eps>0$ a constant $\delta>0$ such that \begin{equation}
			|x-y| \leq \delta \implies |f(x)-f(y)| + |f'(x)-f'(y)| + |f''(x)-f''(y)| \leq \eps. \label{p-eq13}
		\end{equation}
		Fix $\mbb{P} \in \mathfrak{P}_x$. Applying Dynkin's formula, Lemma~\ref{app-1}, we find \begin{align}
			\mbb{E}_{\mbb{P}}f(X_{t \wedge \tau_r^x})-f(x) 
			&= \mbb{E}_{\mbb{P}} \int_{(0,t \wedge \tau_r^x)} A_s^{\mbb{P}} f(X_{s-}) \, ds \label{p-eq15} \\
			&= \mbb{E}_{\mbb{P}} \int_{(0,t \wedge \tau_r^x)} b_s^{\mbb{P}} f'(X_{s-}) \, ds  + \frac{1}{2} \mbb{E}_{\mbb{P}} \int_{(0,t \wedge \tau_r^x)} Q_s^{\mbb{P}} f''(X_{s-}) \, ds \notag \\
			&\quad + \mbb{E}_{\mbb{P}} \int_{(0,t \wedge \tau_r^x)} \int_{y \neq 0} \left( f(X_{s-}+z)-f(X_{s-})-f'(X_{s-}) h(z) \right) \, \nu_s^{\mbb{P}}(dz) \, ds. \notag
		\end{align}
		We are going to estimate \begin{align*}
			J_1 := J_1(t,\mbb{P}) &:= \left| \mbb{E}_{\mbb{P}} \int_{(0,t \wedge \tau_r^x)} b_s^{\mbb{P}} f'(X_{s-}) \,ds - \mbb{E}_{\mbb{P}} \int_{(0,t \wedge \tau_r^x)} b_s^{\mbb{P}} f'(x) \, ds \right| \\
			J_2 := J_2(t,\mbb{P}) &:= \left| \mbb{E}_{\mbb{P}} \int_{(0,t \wedge \tau_r^x)} Q_s^{\mbb{P}} f''(X_{s-}) \,ds - \mbb{E}_{\mbb{P}} \int_{(0,t \wedge \tau_r^x)} Q_s^{\mbb{P}} f''(x) \, ds \right| \\ 
			J_3 := J_3(t,\mbb{P}) &:= \bigg|  \mbb{E}_{\mbb{P}} \int_{(0,t \wedge \tau_r^x)}\int_{y \neq 0} \left( f(X_{s-}+y)-f(X_{s-})-f'(X_{s-}) h(y)\right) \,  \nu_s^{\mbb{P}}(dy) \, ds \\
			&\quad -  \mbb{E}_{\mbb{P}} \int_{(0,t \wedge \tau_r^x)} \int_{y \neq 0} \left( f(x+y)-f(x)-f'(x) h(y) \right) \nu_s^{\mbb{P}}(dy) \, ds\bigg|
		\end{align*}
		for fixed $r>0$, $x \in \mbb{R}^d$. It follows from \eqref{p-eq13}, \eqref{main-eq4} and \eqref{main-eq7} that \begin{align*}
			J_1
			= \mbb{E}_{\mbb{P}} \left| \int_{(0,t \wedge \tau_r^x)} b_s^{\mbb{P}} (f'(X_{s-})-f'(x)) \, ds \right| 
			&\leq  \sup_{\alpha \in A} \sup_{|z-x| \leq r} |b_{\alpha}(z)|  \left[ \eps t + 2t \|f'\|_{\infty} \mbb{P} \left( \sup_{s \leq t} |X_s-x|>\delta \right) \right] \\
			&\leq C_1 t \left[ \eps + \mc{P}^x \left( \sup_{s \leq t} |X_s-x|>\delta \right) \right]
		\end{align*}
		for some constant $C_1=C_1(r,x)<\infty$ which does not depend on $t$ and $\mbb{P} \in \mathfrak{P}_x$. In exactly the same way it follows from the uniform continuity of $f''$, \eqref{main-eq4} and \eqref{main-eq7} that \begin{equation*}
			J_2 \leq C_2 t \left[ \eps + \mc{P}^x \left( \sup_{s \leq t} |X_s-x|>\delta \right) \right]
		\end{equation*}
		for some constant $C_2 = C_2(r,x)$. In order to estimate $J_3$ we note that \begin{align*}
			\Delta := & \left| \big(f(X_{s-}+y)-f(X_{s-})-f'(X_{s-}) y \big) - \big( f(x+y)-f(x)-f'(x)y \big) \right| \\
			&= \left| y \int_0^1 (f'(X_{s-}+\eta y)-f'(X_{s-})) \, d\eta - y \int_0^1 (f'(x+\eta y)-f'(x)) \, d\eta \right| \\
			&= |y|^2 \left| \int_0^1 \int_0^1 \eta f''(X_{s-}+\kappa \eta y) \, d\kappa \, d\eta - \int_0^1 \int_0^1  \eta f''(x+\kappa \eta y) \, d\kappa \, d\eta \right|
		\end{align*}
		and therefore the uniform continuity \eqref{p-eq13} gives \begin{align*}
			\Delta \leq |y|^2 \left(\eps \I_{\{|X_{s-}-x| \leq \delta\}} + 2\|f''\|_{\infty} \I_{\{|X_{s-}-x|>\delta\}} \right).
		\end{align*}
		On the other hand, the uniform continuity of $f$ entails that \begin{align*}
			\left| \big( f(X_{s-}+y)-f(X_{s-})\big)- \big( f(x+y)-f(x) \big) \right|
			&\leq |f(X_{s-}+y)-f(x+y)| + |f(X_{s-})-f(x)| \\
			&\leq 2\eps \I_{\{||X_{s-}-x| \leq \delta\}} + 4\|f\|_{\infty} \I_{\{|X_{s-}-x|>\delta\}}.
		\end{align*}
		Since the truncation function $h$ is bounded, has bounded support and satisfies $h(y)=y$ in a neighborhood of $0$, it follows from the above estimates that \begin{align*}
			 &\left| \big(f(X_{s-}+y)-f(X_{s-})-f'(X_{s-}) h(y) \big) - \big( f(x+y)-f(x)-f'(x)h(y) \big) \right| \\
			 &\leq C_3 \left(\I_{\{|X_{s-}-x| \leq \delta\}} \eps + \I_{\{|X_{s-}-x|>\delta\}} \|f\|_{(2)} \right) \min\{1,|y|^2\}
		\end{align*}
		for some constant $C_3>0$. Combining this with \eqref{main-eq7} we find that \begin{align*}
			J_3
			&\leq C_3 \eps \mbb{E}_{\mbb{P}} \left( \int_{(0,t \wedge \tau_r^x)} \!\! \int_{y \neq 0} \min\{1,|y|^2\} \,  \nu_s^{\mbb{P}}(dy) \, ds\right)\\ &\quad + C_3 \|f\|_{(2)} \mbb{E}_{\mbb{P}} \left( \I_{\{\sup_{s \leq t} |X_s-x|>\delta\}} \int_{(0,t \wedge \tau_r^x)} \int_{y \neq 0} \min\{1,|y|^2\} \,  \nu_s^{\mbb{P}}(dy) \, ds \right) \\
			&\leq C_3' t \sup_{\alpha \in I} \sup_{|z-x| \leq r} \int_{y \neq 0} \min\{1,|y|^2\} \, \nu_{\alpha}(z,dy) \left( \eps + \mbb{P} \left[ \sup_{s \leq t} |X_s-x|>\delta \right] \right) \\
			&\leq C_3'' t \left( \eps + \mc{P}^x \left[ \sup_{s \leq t} |X_s-x|>\delta \right] \right)
		\end{align*}
		for suitable absolute constants $C_3',C_3''>0$. The maximal inequality, Proposition~\ref{max-1}, shows that $\mc{P}^x(\sup_{s \leq t} |X_s-x|>\delta)=O(t)$ as $t \to 0$, and therefore we conclude that \begin{align*}
			&\lim_{t \to 0} \sup_{\mbb{P} \in \mathfrak{P}_x} \frac{1}{t} \left| \mbb{E}_{\mbb{P}} \int_{(0,t \wedge \tau_r^x)} A_s^{\mbb{P}} f(X_{s-}) \, ds - \mbb{E}_{\mbb{P}} \int_{(0,t \wedge \tau_r^x)} A_s^{\mbb{P}} f(x) \, ds \right| \\
			&\leq \limsup_{t \to 0} \sup_{\mbb{P} \in \mathfrak{P}_x} \frac{1}{t} (J_1(t,\mbb{P})+J_2(t,\mbb{P})+J_3(t,\mbb{P}))
			\leq (C_1+C_2+C_3'') \eps \xrightarrow[]{\eps \to 0} 0. 
		\end{align*}
		It now follows from \eqref{p-eq15} that \begin{align*}
			\lim_{t \to 0} \frac{\mc{E}^x f(X_{t \wedge \tau_r^x})-f(x)}{t}
			= \lim_{t \to 0} \sup_{\mbb{P} \in \mathfrak{P}_x} \frac{\mbb{E}_{\mbb{P}}f(X_{t \wedge \tau_r^x})-f(x)}{t} 
			&= \lim_{t \to 0} \sup_{\mbb{P} \in \mathfrak{P}_x} \frac{1}{t} \mbb{E}_{\mbb{P}} \left( \int_{(0,t \wedge \tau_r^x)} A_s^{\mbb{P}} f(X_{s-}) \, ds \right) \\
			&= \lim_{t \to 0} \sup_{\mbb{P} \in \mathfrak{P}_x} \frac{1}{t} \mbb{E}_{\mbb{P}} \left( \int_{(0,t \wedge \tau_r^x)} A_s^{\mbb{P}} f(x) \, ds \right).
		\end{align*}
		\textbf{Step 2: Show that} \begin{equation}
			\lim_{t \to 0} \sup_{\mbb{P} \in \mathfrak{P}_x} \frac{1}{t} \mbb{E}_{\mbb{P}} \left( \int_{(0,t \wedge \tau_r^x)} A_s^{\mbb{P}} f(x) \, ds \right) = \sup_{\alpha \in I} A^{\alpha} f(x). \label{p-eq21}
		\end{equation}
		Proof of \eqref{p-eq21}: Fix $\mbb{P} \in \mathfrak{P}_x$. Because of \eqref{main-eq7} we have \begin{align*}
			\mbb{E}_{\mbb{P}} \int_{(0,t \wedge \tau_r^x)} A_s^{\mbb{P}} f(x) \, ds 
			&\leq \sup_{\alpha \in I} \mbb{E}_{\mbb{P}} \bigg( \int_{(0,t \wedge \tau_r^x)} b_{\alpha}(X_{s})f'(x) \, ds + \frac{1}{2} \int_{(0,t \wedge \tau_r^x)} Q_{\alpha}(X_{s}) f''(x) \, ds \\
			&\qquad + \int_{(0,t \wedge \tau_r^x)} \int_{y \neq 0} \left( f(x+y)-f(x)-f'(x) h(y) \right) \, \nu_{\alpha}(X_{s},dy) \bigg).
		\end{align*}
		Using \eqref{A1}, \eqref{C2} and \eqref{C3} we find that for any $\eps>0$ there exist constants $\delta>0$ and $c_1>0$ such that \begin{align*}
			\mbb{E}_{\mbb{P}} \int_{(0,t \wedge \tau_r^x)} A_s^{\mbb{P}} f(x) \, ds
			&\leq \sup_{\alpha \in I} \mbb{E}_{\mbb{P}} \int_{(0,t \wedge \tau_r^x)} A^{\alpha} f(x) \, ds + t \eps + c_1 M t \|f\|_{(2)}  \mbb{P} \left(\sup_{s \leq t} |X_s-x|>\delta \right)
		\end{align*}
		for \begin{equation*}
			M = M(r,x) =  \sup_{\alpha \in I} \sup_{|z-x| \leq r} \left( |b_{\alpha}(z)| + |Q_{\alpha}(z)| + \int_{y \neq 0} \min\{1,|y|^2\} \, \nu_{\alpha}(z,dy) \right).
		\end{equation*}
		Applying the maximal inequality \eqref{max-eq15} and the dominated convergence theorem we conclude that \begin{align}
			 \limsup_{t \to 0} \sup_{\mbb{P} \in \mathfrak{P}_x} \frac{1}{t} \mbb{E}_{\mbb{P}} \left( \int_{(0,t \wedge \tau_r^x)} A_s^{\mbb{P}} f(x) \, ds \right)
			 &\leq \sup_{\alpha \in I} A^{\alpha} f(x) + \eps
			 \xrightarrow[]{\eps \to 0} \sup_{\alpha \in I} A^{\alpha} f(x). \label{p-eq25}
		\end{align}
		On the other hand \eqref{C5} gives for any $\alpha \in I$ \begin{align*}
			 \sup_{\mbb{P} \in \mathfrak{P}_x} \frac{1}{t} \mbb{E}_{\mbb{P}} \left( \int_{(0,t \wedge \tau_r^x)} A_s^{\mbb{P}} f(x) \, ds \right)
			 &\geq \frac{1}{t} \mbb{E}_{\mbb{P}^{\alpha}} \left( \int_{(0,t \wedge \tau_r^x)} A_s^{\mbb{P}_{\alpha}} f(x) \, ds \right) \\
			 &= \mbb{E}_{\mbb{P}_{\alpha}} \bigg( \int_{(0,t \wedge \tau_r^x)} b_{\alpha}(X_{s-})f'(x) \, ds + \frac{1}{2} \int_{(0,t \wedge \tau_r^x)} Q_{\alpha}(X_{s-}) f''(x) \, ds \\
			 			&\qquad + \int_{(0,t \wedge \tau_r^x)} \int_{y \neq 0} \left( f(x+y)-f(x)-f'(x) h(y) \right) \, \nu_{\alpha}(X_{s-},dy) \bigg).
		\end{align*}
		Invoking once more the uniform equicontinuity and the boundedness of the coefficients on compacts we can choose for any $\eps>0$ some constants $\delta>0$, $c_2>0$ such that \begin{align*}
			 \sup_{\mbb{P} \in \mathfrak{P}_x} \frac{1}{t} \mbb{E}_{\mbb{P}} \left( \int_{(0,t \wedge \tau_r^x)} A_s^{\mbb{P}} f(x) \, ds \right)
			 &\geq \mbb{E}_{\mbb{P}^{\alpha}} \int_{(0,t \wedge \tau_r^x)} A^{\alpha} f(x) \, ds - t \eps - c_2 M t \|f\|_{(2)}  \mbb{P}^{\alpha} \left( \sup_{s \leq t} |X_s-x|>\delta \right),
		\end{align*}
		and the maximal inequality \eqref{max-eq15} gives \begin{equation*}
			\liminf_{t \to 0}\sup_{\mbb{P} \in \mathfrak{P}_x} \frac{1}{t} \mbb{E}_{\mbb{P}} \left( \int_{(0,t \wedge \tau_r^x)} A_s^{\mbb{P}} f(x) \, ds \right)
			\geq A^{\alpha} f(x)- \eps \xrightarrow[]{\eps \to 0} A^{\alpha} f(x).
		\end{equation*}
		As $\alpha \in I$ is arbitrary this proves that \begin{equation}
			\liminf_{t \to 0} \sup_{\mbb{P} \in \mathfrak{P}_x} \frac{1}{t} \mbb{E}_{\mbb{P}} \left( \int_{(0,t \wedge \tau_r^x)} A_s^{\mbb{P}} f(x) \, ds \right)
			\geq \sup_{\alpha \in I} A^{\alpha} f(x). \label{p-eq29}
		\end{equation}
		Combining \eqref{p-eq25} and \eqref{p-eq29} gives \eqref{p-eq21}. \par
		\textbf{Step 3: Show that there exists $c_r = c_r(x)>0$ such that $c_r \to 0$ as $r \to \infty$ and} \begin{equation}
			 \sup_{\mbb{P} \in \mathfrak{P}_x} \left| \frac{\mbb{E}_{\mbb{P}} f(X_t)-\mbb{E}_{\mbb{P}}f(X_{t \wedge \tau_r^x})}{t} \right| \leq c_r t \fa  t \geq 0,\,  r>0. \label{p-eq30}
		\end{equation}
		Proof of \eqref{p-eq30}: As $f(X_t)-f(X_{t \wedge \tau_r^x})=0$ on $\{\tau_r^x \geq t\}$ we have \begin{align*}
			|\mbb{E}_{\mbb{P}} f(X_t)-\mbb{E}_{\mbb{P}}f(X_{t \wedge \tau_r^x})|
			&\leq 2\|f\|_{\infty} \mbb{P} \left( \sup_{s \leq t} |X_s-x|>r \right).
		\end{align*}
		Applying the maximal inequality we obtain that \begin{equation*}
			|\mbb{E}_{\mbb{P}} f(X_t)-\mbb{E}_{\mbb{P}}f(X_{t \wedge \tau_r^x})|
			\leq c_r t 
		\end{equation*}
		for \begin{equation*}
			c_r := 2 \|f\|_{\infty} c \sup_{|z-x| \leq r} \sup_{|\xi| \leq r^{-1}} \sup_{\alpha \in I} |q_{\alpha}(z,\xi)|
		\end{equation*}
		where $c>0$ is some absolute constant; by \eqref{A1}, $c_r \to 0$ as $r \to \infty$. \par
		\textbf{Step 4: Conclusion.} Using the elementary estimate \begin{equation*}
			\sup_{\mbb{P}} a_{\mbb{P}}  - \sup_{\mbb{P}} |b_{\mbb{P}}| \leq \sup_{\mbb{P}} (a_{\mbb{P}}+b_{\mbb{P}}) \leq \sup_{\mbb{P}} a_{\mbb{P}} + \sup_{\mbb{P}} |b_{\mbb{P}}|
		\end{equation*}
		and the fact that \begin{align*}
			\frac{\mc{E}^x f(X_t)-f(x)}{t}&  - \sup_{\alpha \in I} A^{\alpha} f(x) \\
			&= \sup_{\mbb{P} \in \mathfrak{P}_x} \left( \left[ \frac{\mbb{E}_{\mbb{P}} f(X_{t \wedge \tau_r^x})-f(x)}{t} - \sup_{\alpha \in I} A^{\alpha} f(x) \right] + \frac{\mbb{E}_{\mbb{P}} f(X_t)-\mbb{E}_{\mbb{P}} f(X_{t \wedge \tau_r^x})}{t} \right)
		\end{align*}
		it follows easily from Step 1-3 that \begin{equation*}
			\lim_{t \to 0} \frac{\mc{E}^x f(X_t)-f(x)}{t} = \sup_{\alpha \in I} A^{\alpha} f(x) \fa  x \in \mbb{R}^d;
		\end{equation*}
		this proves the assertion for the case that the family of L\'evy triplets satisfies \eqref{A1}. If the family of L\'evy triplets is uniformly bounded, i.\,e.\ if it satisfies \eqref{A2}, we replace in the above reasoning the stopped process $(X_{t \wedge \tau_r^x})_{t \geq 0}$ by the original process $(X_t)_{t \geq 0}$; formally this corresponds to setting $r:= \infty$ (note that the sample paths of $(X_t)_{t \geq 0}$ do not explode in finite time). In particular, Step 3 can be omitted since the expression on the left-hand side of \eqref{p-eq30} equals zero. 
\end{proof}

We are now ready to prove our main result, Theorem~\ref{main-3}.

\begin{proof}[Proof of Theorem~\ref{main-3}]
	Let $f \in \mc{H}$ be a function such that $u(t,x) := T_t f(x)$ depends continuously on $(t,x)$. If we denote by $A$ the generator of the sublinear Markov semigroup $(T_t)_{t \geq 0}$, then we find by applying Theorem~\ref{pre-33} that $u$ is a viscosity solution (in the sense of Definition~\ref{pre-21}) to the evolution equation \begin{equation*}
		\partial_t u(t,x) - A_x u(t,x) = 0 \qquad u(0,x)=f(x).
	\end{equation*}
	Since Lemma~\ref{p-1} shows that \begin{equation*}
		A \varphi(x) = \sup_{\alpha \in I} q_{\alpha}(x,D) \varphi(x)
	\end{equation*}
	for any $\varphi \in C_b^{\infty}(\mbb{R}^d)$, we conclude that $u$ is a viscosity solution to \begin{equation*}
		\partial_t u(t,x) - \sup_{\alpha \in I} q_{\alpha}(x,D) u(t,x) = 0. \qedhere
	\end{equation*}
\end{proof}

\begin{proof}[Proof of Proposition~\ref{main-6}]
	It was shown in \cite{julian} that $(T_t)_{t \geq 0}$ is a sublinear Markov semigroup on the space of bounded upper semi-analytic functions and that $(T_t)_{t \geq 0}$ is spatially homogeneous, i.\,e.\ \begin{equation*}
		T_t f(x) = \mc{E}^x f(X_t) = \mc{E}^0 f(x+X_t), \qquad t \geq 0, \, x \in \mbb{R}^d 
	\end{equation*}
	for any bounded upper semi-analytic function $f$. In order to show that $(T_t)_{t \geq 0}$ is a sublinear Markov semigroup on the spaces \eqref{main-6-ii}-\eqref{main-6-iv}, it therefore suffices to prove that the semigroup leaves each of the spaces invariant. The spatial homogeneity of the semigroup gives  \begin{align*}
		T_t f(x)-T_t f(y) = \mc{E}^0 f(x+X_t)- \mc{E}^0 f(y+X_t) &\leq \mc{E}^0 \big( f(x+X_t)-f(y+X_t) \big) \\ &\leq \mc{E}^0 |f(x+X_t)-f(y+X_t)| 
	\end{align*}
	where we have used for the first inequality the subaddivity of $\mc{P}^0$. Interchanging the roles of $x$ and $y$ we obtain that \begin{align}
		|T_t f(x)-T_t f(y)| \leq \mc{E}^0|f(x+X_t)-f(y+X_t)|, \label{p-eq33}
	\end{align}
	and this shows that $x \mapsto T_t f(x)$ is Lipschitz continuous (resp.\ uniformly continuous) if $f$ is Lipschitz continuous (resp.\ Lipschitz continuous). Now assume additionally that \eqref{main-eq10} holds; to finish the proof we have to show that $T_t f$ is continuous for any $f \in C_b(\mbb{R}^d)$. Fix $f \in C_b(\mbb{R}^d)$, $\eps>0$ and $x,y \in B(0,k)$. It follows from \eqref{p-eq33} that \begin{align*}
		|T_t f(x)-T_t f(y)|
		\leq \sup_{\substack{|u-v| \leq |x-y| \\ u,v \in B(0,k+R)}} |f(u)-f(v)| + 2 \|f\|_{\infty} \mc{P}^0 \left( \sup_{s \leq t} |X_s| >R \right),
	\end{align*}
	and therefore the maximal inequality \eqref{max-eq15} gives \begin{align}
		|T_t f(x)-T_t f(y)|
		\leq \sup_{\substack{|u-v| \leq |x-y| \\ u,v \in B(0,k+R)}} |f(u)-f(v)| + ct \|f\|_{\infty} \sup_{\alpha \in I} \sup_{|\xi| \leq R^{-1}} |\psi_{\alpha}(\xi)| \label{p-eq34}
	\end{align}
	for some absolute constant $c>0$ and \begin{equation*}
		\psi_{\alpha}(\xi) = -ib_{\alpha} \cdot \xi + \frac{1}{2} \xi \cdot Q_{\alpha} \xi + \int_{y \neq 0} \left( 1-e^{iy \cdot \xi} + i\xi \cdot h(y) \right) \, \nu_{\alpha}(dy).
	\end{equation*}
	Because of the tightness condition \eqref{main-eq10}, we can choose $R>0$ sufficiently large the second term on the right-hand side of \eqref{p-eq34} is smaller than $\eps$, cf.\ Lemma~\ref{app-3}. Since $f$ is uniformly continuous on $B(0,k+R)$ the first term on the right-hand side of \eqref{p-eq34} is smaller than $\eps$ for $|x-y| \leq \delta$ small, and this proves the continuity of $x \mapsto T_t f(x)$.
\end{proof}

\begin{proof}[Proof of Corollary~\ref{main-7}] 
	To prove that $u(t,x) = T_t f(x)$ is a solution, we are going to apply Theorem~\ref{main-3}. The uniform boundedness condition \eqref{A2} holds by assumption whereas \eqref{C2},\eqref{C3} are trivially satisfied. Condition \eqref{C4} follows from the Proposition~\ref{main-6}, and \eqref{C5} is clearly satisfied since we can choose $\mbb{P}^{\alpha} \in \mathfrak{P}_x$ such that the canonical process $(X_t)_{t \geq 0}$ is a (classical) L\'evy process with characteristic exponent $\psi_{\alpha}$ on the Skorohod space endowed with the probability meausre $\mbb{P}^{\alpha}$. In order to apply Theorem~\ref{main-3}, we finally note that $(t,x) \mapsto T_t f(x)$ is continuous if \eqref{main-7-i} or \eqref{main-7-ii} holds; this follows readily from Proposition~\ref{main-6} and Theorem~\ref{max-5}. \par
	If additionally the tightness condition \eqref{main-eq17} is satisfied, then \cite[Corollary 2.34]{julian} shows that the solution is unique.
\end{proof}

\begin{proof}[Proof of Corollary~\ref{main-11}] 
	We are going to apply Theorem~\ref{main-3} to prove the assertion. It was shown in \cite[Proposition 4.60]{julian} that \begin{equation*}
		P_t f(x) = \mc{E}^0 f(Z_t^x) = \sup_{\mbb{P} \in \mathfrak{P}_0} \mbb{E}_{\mbb{P}} f(Z_t^x), \qquad t \geq 0, \, x \in \mbb{R}^d,
	\end{equation*}
	defines a sublinear Markov semigroup on the space of bounded upper semi-analytic functions. Fix a Lipschitz continuous truncation function $\hat{h}: \mbb{R}^k \to \mbb{R}^k$. For each $\mbb{P} \in \mathfrak{P}_0$ the process $(Z_t^x)_{t \geq 0}$ is a $\mbb{P}$-semimartingale with differential characteristics $(b^{\mbb{P}}_s, Q_s^{\mbb{P}},\nu_s^{\mbb{P}})$ with respect to the truncation function $\hat{h}$ satisfying \begin{equation*}
		(b_s^{\mbb{P}},Q_s^{\mbb{P}},\nu_s^{\mbb{P}}) \in \bigcup_{\alpha \in I} \{(\hat{b}_{\alpha}(Z_s),\hat{Q}_{\alpha}(Z_s),\hat{\nu}_{\alpha}(Z_s))\} \quad \lambda(ds)\times\mbb{P}\text{-a.s.}
	\end{equation*}
	for the uncertainty coefficients \begin{align*} 
		\hat{b}_{\alpha}(x) &:= \sigma(x) \cdot b_{\alpha} - \int_{y \neq 0} \left(\sigma(x) \cdot h(y)-\hat{h}(\sigma(x) \cdot y)\right) \, \nu_{\alpha}(dy) \\
		\hat{Q}_{\alpha}(x) &:= \sigma(x) Q_{\alpha} \sigma(x)^T \\
		\hat{\nu}_{\alpha}(x,B) &:= \int_{y \neq 0} \I_B(\sigma(x) \cdot y) \, \nu_{\alpha}(dy),
	\end{align*}
	see e.\,g.\ \cite[Proposition 9.5.3]{jacod}; recall that $h$ is the truncation function associated with the driving L\'evy process for sublinear expectations. Consequently, the push-forward $\hat{\mc{P}}_x := \mathfrak{P}_0 \circ (Z^x)^{-1}$ satisfies \begin{equation*}
		\hat{\mc{P}}_x \subseteq \left\{ \hat{\mbb{P}} \in \mathfrak{P}_{\sem}^{\ac}(D_x); (b_s^{\hat{\mbb{P}}},Q_s^{\hat{\mbb{P}}},\nu_s^{\hat{\mbb{P}}})(\omega) \in \bigcup_{\alpha \in I} \{(\hat{b}_{\alpha},\hat{Q}_{\alpha},\hat{\nu}_{\alpha})(\omega(s))\} \, \, \lambda(ds)\times \hat{\mbb{P}}(d\omega)-\text{a.s.}\right\};
	\end{equation*}
	this proves that assumption \eqref{C4} of Theorem~\ref{main-3} is satisfied. If $\mbb{P}_{\alpha} \in \mathfrak{P}_0$ is such that $(X_t)_{t \geq 0}$ is a classical L\'evy process with L\'evy triplet $(b_{\alpha},Q_{\alpha},\nu_{\alpha})$ with respect to $\mbb{P}_{\alpha}$, then the differential characteristics associated to the measure $\hat{\mbb{P}}_{\alpha} := \mbb{P}_{\alpha} \circ (Z^x)^{-1} \in \hat{\mc{P}}_x$ is given by \begin{equation*}
		(b_s^{\hat{\mbb{P}}_{\alpha}},Q_s^{\hat{\mbb{P}}_{\alpha}},\nu_s^{\hat{\mbb{P}}_{\alpha}})(\omega) = (\hat{b}_{\alpha}(\omega(s)),\hat{Q}_{\alpha}(\omega(s)),\hat{\nu}_{\alpha}(\omega(s)))
	\end{equation*}
	$\lambda(ds)\times\hat{\mbb{P}}_{\alpha}(\omega)$-almost surely, and this gives \eqref{C5}. The continuity assumptions \eqref{C2} and \eqref{C3} are a consequence of the tightness condition \eqref{main-eq17} and the Lipschitz continuity of $\sigma$ and $\hat{h}$. 
	Moreover, we clearly have \begin{equation*}
		\sup_{\alpha \in I} \sup_{x \in K} \left( |\hat{b}_{\alpha}(x)| + |\hat{Q}_{\alpha}(x)| + \int_{y \neq 0} \min\{1,|y|^2\} \, \hat{\nu}_{\alpha}(x,dy) \right) < \infty
	\end{equation*}
	for any compact set $K \subseteq \mbb{R}^k$.  The continuous negative definite function $\hat{q}_{\alpha}(x,\cdot)$ associated with $(\hat{b}_{\alpha}(x),\hat{Q}_{\alpha}(x),\hat{\nu}_{\alpha}(x,dy))$ via the L\'evy--Khintchine formula equals  \begin{equation*}
		\hat{q}_{\alpha}(x,\xi) = \psi_{\alpha}(\sigma(x)^T \xi), \qquad \alpha \in I, \, x,\xi \in \mbb{R}^k,
	\end{equation*}
	cf.\ \cite{schnurr}, and so the sublinear growth condition on $\sigma$ entails that \begin{align*}
		\sup_{\alpha \in I} \sup_{|x| \leq 2r} \sup_{|\xi| \leq r^{-1}} |\hat{q}_{\alpha}(x,\xi)|
		&\leq \sup_{\alpha \in I} \sup_{|\eta| \leq c r^{-\eps}} |\psi_{\alpha}(\eta)|
	\end{align*}
	for suitable constants $c>0$ and $\eps \in (0,1]$. By the tightness and uniform boundedness of the family $(b_{\alpha},Q_{\alpha},\nu_{\alpha})$, $\alpha \in I$, the right-hand side converges to $0$ as $r \to \infty$, see Lemma~\ref{app-3} for details. Consequently, we have shown that condition \eqref{A1} holds. \par
	Now let $f \in C_b(\mbb{R}^k)$. In order to apply Theorem~\ref{main-3}, it remains to prove that \begin{equation*}
		(t,x) \mapsto P_t f(x) = \mc{E}^0 f(Z_t^x)
	\end{equation*}
	is continuous. To this end, we note that Theorem~\ref{max-5}\eqref{max-5-iii} yields that $t \mapsto P_t f(x)$, $x \in K$, is equi-continuous for any compact set $K \subseteq \mbb{R}^k$. On the other hand, \cite[Proposition 4.61]{julian} shows that $x \mapsto P_t f(x)$ is continuous for each fixed $t \geq 0$. Combining both facts we conclude that $(t,x) \mapsto P_t f(x)$ is continuous. \par
	An application of Theorem~\ref{main-3} now shows that $u(t,x) = P_t f(x)$ is a viscosity solution to \eqref{main-eq25}. The uniqueness follows from \cite[Corollary 2.34]{julian}.
\end{proof}

\begin{proof}[Proof of Corollary~\ref{main-13}]
	We apply Theorem~\ref{main-3} for $\mathfrak{P}_x := \{\mbb{P}^x\}$ and the family of L\'evy triplets $(b(x),Q(x),\nu(x,dy))$. First of all, we note that \cite[Lemma 6.2]{schnurr} (see also \cite[Theorem 2.31]{ltp}) gives \begin{equation*}
		\sup_{x \in K} \left( |b(x)| + |Q(x)| + \int_{y \neq 0} \min\{1,|y|^2\} \, \nu(x,dy) \right) < \infty
	\end{equation*}
	for any compact set $K \subseteq \mbb{R}^d$. The assumptions of Corollary~\ref{main-13} are tailored in such a way that one of the  conditions \eqref{A1},\eqref{A2} is satisfied. Moreover, the (classical) Markov semigroup $T_t f(x) = \mbb{E}^x f(X_t)$ is clearly a sublinear Markov semigroup on the Borel measurable bounded functions, and so \eqref{C3} holds for $\mc{H} := \mc{B}_b(\mbb{R}^d)$. Condition \eqref{C4} is automatically satisfied since the Feller process $(X_t)_{t \geq 0}$ is a semimartingale with differential characteristics $(b(X_{s-}), Q(X_{s-}), \nu(X_{s-},dy))$. Applying Theorem~\ref{main-3} we find that $u(t,x) := T_t f(x)$ is a viscosity solution to \eqref{main-eq31} for any $f \in \mc{B}_b(\mbb{R}^d)$ such that $(t,x) \mapsto u(t,x)$ is continuous. We consider the two cases separately: \begin{enumerate}
		\item $f \in C_{\infty}(\mbb{R}^d)$ and $q$ has bounded coefficients:  The Feller property gives that $x \mapsto T_t f(x)$ is continuous. On the other hand, Theorem~\ref{max-5}\eqref{max-5-ii} shows that $t \mapsto T_t f(x)$ is continuous uniformly in $x \in \mbb{R}^d$. Therefore we infer that $(t,x) \mapsto T_t f(x)$ is continuous for any $f \in C_{\infty}(\mbb{R}^d)$, and combining this with the first part of the proof this proves that $u$ is a viscosity solution to \eqref{main-eq31}. 
		\item $f \in C_b(\mbb{R}^d)$ and the uniform continuity condition \eqref{main-eq29} holds: \cite[Theorem 5.5]{cons}, see also \cite[Lemma 3.3]{perpetual}, shows that $(X_t)_{t \geq 0}$ is conservative, i.\,e.\ $T_t 1=1$ for any $t \geq 0$. Since \begin{equation*}
			T_t(C_{\infty}(\mbb{R}^d)) \subseteq C_{\infty}(\mbb{R}^d), T_t(1) \in C_b(\mbb{R}^d) \implies T_t(C_b(\mbb{R}^d)) \subseteq C_b(\mbb{R}^d),
		\end{equation*}
		cf.\ \cite[Section 3]{cons} or \cite[Theorem 1.9]{ltp}, we find that $x \mapsto T_t f(x)$ is continuous for all $t \geq 0$. Moreover, Theorem~\ref{max-5}\eqref{max-5-iii} gives that $t \mapsto T_t f(x)$ is continuous uniformly in $x \in K$ for any compact set $K \subseteq \mbb{R}^d$. Consequently, we obtain that $(t,x) \mapsto T_t f(x)$ is continuous, and by the first part of the proof $u(t,x)=T_t f(x)$ is a viscosity solution to \eqref{main-eq31}. \qedhere
	\end{enumerate}
\end{proof}

\appendix

\section{} \label{app}

\begin{lem}[Dynkin formula] \label{app-1} 
Let $(X_t)_{t \geq 0}$ be an $(\mc{F}_t)_{t \geq 0}$-adapted \cadlag semimartingale on a (classical) probability space $(\Omega,\mc{A},\mbb{P})$ started at $X_0 = x \in \mbb{R}^d$. Assume that $(X_t)_{t \geq 0}$ has differential characteristics $(b_s,Q_s,\nu_s)$ satisfying \begin{equation}
	(b_s,Q_s,\nu_s)(\omega) \in \bigcup_{\alpha \in I} \{(b_{\alpha},Q_{\alpha},\nu_{\alpha})(X_s(\omega))\} \, \, \lambda(ds)\times \mbb{P}(\omega)\text{-a.s.} \label{app-eq3}
\end{equation}
for a family of L\'evy triplets $(b_{\alpha}(z),Q_{\alpha}(z),\nu_{\alpha}(z))$, $\alpha \in I$, $z \in \mbb{R}^d$. \begin{enumerate}
	\item\label{app-1-i} If \begin{equation*}
			K_n := \sup_{\alpha \in I} \sup_{|z-x| \leq n} \left( |b_{\alpha}(z)|+ |Q_{\alpha}(z)| + \int_{y \neq 0} \min\{1,|y|^2\} \, \nu_{\alpha}(z,dy) \right)<\infty
	\end{equation*}
	for all $n \in \mbb{N}$, then \begin{equation*}
		\mbb{E}f(X_{t \wedge \tau_r}) - f(x) = \mbb{E} \left( \int_{(0,t \wedge \tau_r)} A_s f(X_{s-}) \, ds \right)
	\end{equation*}
	for any $r>0$ and $f \in C_b^2(\mbb{R}^d)$; here $\tau_r$ denotes the first exit time from $\overline{B(x,r)}$ and 
	\begin{align*}
		A_s f(z) := b_s \cdot \nabla f(z) + \frac{1}{2} \tr(Q_s \nabla^2 f(z)) + \int_{y \neq 0} \big( f(z+y)-f(z)-\nabla f(z) \cdot  h(y) \big)  \, \nu_s(dy). 
	\end{align*}
	\item\label{app-1-ii} If $K := \sup_{n \in \mbb{N}} K_n < \infty$, then \begin{equation*}
		\mbb{E}f(X_t) - f(x) = \mbb{E} \left( \int_{(0,t)} A_s f(X_{s-}) \, ds \right)
	\end{equation*}
	for any $t \geq 0$ and $f \in C_b^2(\mbb{R}^d)$. 
\end{enumerate}
\end{lem}

\begin{proof}
	First of all, we note that we may replace $\mc{F}_t$ by the augmented filtration $\mc{F}_t^{\mbb{P}}$; this follows from the fact that $(X_t,\mc{F}_t^{\mbb{P}})$ is a semimartingale (on the completed probability space) which has almost surely the same characteristics as $(X_t,\mc{F}_t)_{t \geq 0}$, cf.\ \cite[Proposition 2.2]{nutz2}. Since the first exit time $\tau_r$ is an $\mc{F}_t^{\mbb{P}}$-stopping time, this allows us to use standard stopping techniques. By \cite[Theorem II.2.34]{jacod}, the semimartingale $(X_t)_{t \geq 0}$ has a canonical representation of the form \begin{equation*}
		X_t = \int_0^t b_s \, ds + X_t^{c} + X_t^{d} + \int_0^t \!\! \int (y-h(y)) \, \mu^X(dy,ds)
	\end{equation*}
	where $(X_t^{c})_{t \geq 0}$ is the continuous local martingale part of $(X_t)_{t \geq 0}$, $\nu_s(dy) \, ds$ is the compensator of the jump measure $\mu^X(dy,ds)$ and \begin{equation*}
		X_t^{d} = \int_0^t\!\! \int h(y) \, (\mu^X(dy,ds)-\nu_s(dy) \, ds)
	\end{equation*}
	is the purely discontinuous local martingale part. An application of It\^o's formula for semimartingales, see e.\,g.\ \cite[Theorem II.2.42]{jacod}, shows that \begin{align*}
		f(X_t)-f(X_0)- \int_0^t A_s f(X_{s-}) \, ds
		&= M_t^c + M_t^d
	\end{align*}
	where \begin{align*}
		M_t^c &:= \int_0^t \nabla f(X_{s-}) \, dX_s^{c,\mbb{P}} \\
		M_t^d &:= \int_0^t \!\! \int \big(f(X_{s-}+h(y))-f(X_{s-})\big) \, (\mu^X(ds,dy)-\nu_s(dy) \, ds)
	\end{align*}
	are local martingales. By \eqref{app-eq3} the quadratic variation of the stopped processes $(M_{t \wedge \tau_r}^c)_{t \geq 0}$ and $(M_{t \wedge \tau_r}^d)_{t \geq 0}$ satisfy \begin{align*}
		\mbb{E}[M_{\tau_r}^c]_t &\leq t \|\nabla f\|_{\infty}^2  \sup_{\alpha \in I} \sup_{|z-x| \leq r} |Q_{\alpha}(z)|^2 \\
		\mbb{E}[M_{\tau_r}^c]_t &\leq Ct \|f\|_{(2)} \sup_{\alpha \in I} \sup_{|z-x| \leq r} \int_{y \neq 0} \min\{1,|y|^2\} \, \nu_{\alpha}(z,dy)
	\end{align*}
	for a suitable constant $C>0$. If the constant $K_n$, defined in \eqref{app-1-i}, is finite for each $n \in \mbb{N}$, then the expressions on the right-hand side are finite, and \cite[Corollary II.6.3]{protter} gives that $(M_{t \wedge \tau_r}^c)_{t \geq 0}$ and $(M_{t \wedge \tau_r}^d)_{t \geq 0}$ are martingales which proves \eqref{app-1-i}. For \eqref{app-1-ii} we can use a very similar reasoning (formally we can set $r:=\infty$).
\end{proof}

\begin{lem} \label{app-3}
	Let $\psi_{\alpha}: \mbb{R}^d \to \mbb{C}$, $\alpha \in I$, be a family of continuous negative definite functions with characteristics $(b_{\alpha},Q_{\alpha},\nu_{\alpha})$, $\alpha \in I$, with respect to a truncation function $h$. If the family $(b_{\alpha},Q_{\alpha},\nu_{\alpha})$ is uniformly bounded, i.\,e.\ \begin{equation*}
		M:= \sup_{\alpha \in I} \left( |b_{\alpha}| + |Q_{\alpha}| + \int_{y \neq 0} \min\{1,|y|^2\} \, \nu_{\alpha}(dy) \right)<\infty,
	\end{equation*}
	and satisfies the tightness condition \begin{equation*}
		\lim_{R \to \infty} \sup_{\alpha \in I} \int_{|y|>R} \, \nu_{\alpha}(dy)=0,
	\end{equation*}
	then \begin{equation*}
		\lim_{r \to 0} \sup_{\alpha \in I} \sup_{|\xi| \leq r} |\psi_{\alpha}(\xi)|=0.
	\end{equation*}
\end{lem}

For the particular case $I=\mbb{R}^d$ the result follows from \cite[Proof of Theorem 4.4]{cons}; since the proof does not rely on the topological structure of the Euclidean space, we may replace $\mbb{R}^d$ by an arbitrary index set $I$.

\end{document}